[1994/12/01]
\documentclass{amsart}
\usepackage{amsmath,amsthm,amsfonts,amsopn,verbatim}

\chardef\bslash=`\\ 

\newtheorem{thm}{Theorem}[section]

\newtheorem{lem}[thm]{Lemma}
\newtheorem{prop}[thm]{Proposition}

\newtheorem*{thm*}{Theorem}
\newtheorem*{lem*}{Lemma}
\newtheorem*{prop*}{Proposition}
\theoremstyle{definition}
\newtheorem{defn}{Definition}[section]

\theoremstyle{remark}
\newtheorem{rem}{Remark}[section]

\usepackage{amscd,amsthm}
\usepackage{}

\newcommand{\eval}[2][\right]{\relax
  \ifx#1\right\relax \left.\fi#2#1\rvert}

\begin{document}
\title[
Persistence of H\"{o}lder continuity 
for  integro-differential  equations]{Persistence of H\"{o}lder continuity 
\\for non-local integro-differential  equations}
\author{Kyudong Choi}
\address{1 University Station C1200, Austin, TX 78712, USA}
\thanks{Department of Mathematics, University of Texas at Austin}
\email{kchoi@math.utexas.edu}
\renewcommand{\sectionmark}[1]{}
\begin{abstract}
In this paper, we consider
  non-local integro-differential equations
  under certain natural assumptions
on the kernel, and obtain
 persistence of H\"{o}lder continuity 
for their solutions. In other words, we prove that a solution  stays in $C^\beta$ for all time 
 if
its initial data  lies in $C^\beta$. This result has an application
for a fully non-linear problem, which is used in the field of image processing.
The proof  is in the spirit of \cite{variation_kiselev} where
Kiselev and Nazarov
 established H\"{o}lder continuity
of the critical surface quasi-geostrophic (SQG) equation.
\end{abstract}
\maketitle
\section{Introduction and the main result}\label{sec_Introduction}

Let $N\geq1$ be any dimension.
We consider the following evolution equation
\begin{equation}\label{main_lin_eq}
  \partial_t w(t,x)=\int_{\mathbb{R}^N}[w(t,y)-w(t,x)]K(t,x,y)dy
\end{equation} 
where $K$ satisfies 
the $weak$-$(*)$-kernel condition,
which will be given in Definition \ref{def_parameter2}.
The above integral is understood in the sense of principal value.
More precisely,
 we denote the integral operator $T^K_t$ and $(T^K_t)_\epsilon$ for $\epsilon>0$ corresponding 
to any given kernel $K$
at time $t$ by 
\begin{equation*}\begin{split}
 (T^K_t)_\epsilon(f)(x)&=\int_{|x-y|\geq\epsilon}[f(x)-f(y)]K(t,x,y)dy
 \quad\mbox{ and}\\
 (T^K_t)(f)(x)&=\lim\limits_{\epsilon\rightarrow 0}(T^K_t)_\epsilon(f)(x).
\end{split}\end{equation*}
Then, \eqref{main_lin_eq} is equivalent to $(\partial_t w)(t,x)
+T^K_t(w(t,\cdot))(x)=0$. Related to the above singular integral, 
there have been many interests recently, not only from the field of analysis, but also from the field of probability
(e.g. Caffarelli  and Silvestre \cite{caffa_silve_regular},
Schwab \cite{MR2733264}, 
Bass  and Levin \cite{MR1895210}, 
Jacob, Potrykus,   and Wu  \cite{MR2728173},
 and Chen,  Kim,   and Kumagai  \cite{MR2806700}).\\




Our main concern   is to obtain \textit{a priori} estimate
for solutions of \eqref{main_lin_eq}. 
The aim is to prove the result \cite{ccv} of
Caffarelli, Chan,  and Vasseur with
different techniques (a similar result for the stationary case was obtained by Kassmann in \cite{kassman_apriori}).
In particular, we prove
 persistence of H\"{o}lder continuity in
 $L^\infty(0,\infty; C^\beta(\mathbb{R}^N))$,
 which is a new result,
by observing the evolution of a dual class of test functions. This class, which appears in the work of
 Kiselev and Nazarov \cite{variation_kiselev},
 plays a similar role of the dual space of $C^\beta$. 
They obtained, in \cite{variation_kiselev},
H\"{o}lder regularity for solutions of
the critical surface quasi-geostrophic (SQG) 
equation.
It is interesting
to compare this method with that of 
Caffarelli  and Vasseur \cite{caf_vas}. In \cite{caf_vas}, the estimate 
$C^\beta([t,\infty)\times\mathbb{R}^N))$
for any $t>0$
was proved by using 
a De Giorgi iteration technique 
(for other different proofs, we refer to 
Kiselev, Nazarov,  and Volberg \cite{knv} and Constantin and Vicol
\cite{constantin_vicol}).\\

We define the $(*)$-kernel condition
 on the kernel $K$.
\begin{defn}\label{def_parameter} 
Let $0<{\alpha}<2$, $0<\zeta\leq\infty$, 
$0\leq\omega<{\alpha}$
 and 
 $1\leq\Lambda<\infty$ 
(for the case $\alpha\geq1$, 
three more parameters $\nu$, $s_0$ and $\tau$, which 
are satisfying $(\alpha-1)<\nu<1$,
$0<s_0\leq\infty$,
$0\leq\tau<\infty$ and $\nu+\omega<\min\{N,\alpha\}$,
are needed).
Then we say that a 
measurable
 function
 $K: [0,T]\times\mathbb{R}^N\times\mathbb{R}^N
\rightarrow [0,\infty)$
satisfies the $(*)$-kernel condition on $[0,T]$ 
for the parameter set $\{{\alpha},\zeta,\omega,
\Lambda\}$ 
(if $\alpha\geq1$, 
for the parameter set $\{{\alpha},\zeta,\omega,
\Lambda,\nu,s_0,\tau\}$)
if the following conditions hold
for all finite $t\in[0,T]$:
\begin{equation}\label{cond_symm}
\odot
\mbox{ Symmetry in } x,y \mbox{: } K(t,x,y)=K(t,y,x) ,
\mbox{ for   } x,y\in\mathbb{R}^N.
\end{equation}
\begin{equation}\label{cond_bounds}
\odot
\mbox{ Bounds: }{\Lambda^{-1}}\cdot\mathbf{1}_{|x-y|\leq \zeta}
\leq K(t,x,y)|x-y|^{N+{\alpha}}
\leq \Lambda\cdot(1+ |x-y|^\omega)
\quad\mbox{ for   } x,y\in\mathbb{R}^N.
\end{equation}
For convenience, we define the associated function $k$ 
by $k(t,x,z)=K(t,x,x+z)|z|^{N+{\alpha}}$. 
Then the above two conditions are equivalent to
$k(t,x,y-x)=k(t,y,x-y)$ and 
${\Lambda^{-1}}\cdot\mathbf{1}_{|z|\leq\zeta}\leq k(t,x,z)
\leq \Lambda\cdot(1+|z|^\omega) $, respectively.\\

Only when $\alpha\geq1$,   we assume one more condition:
\begin{equation}\label{cond_Holder}
\odot
\mbox{ Local H\"{o}lder continuity: }\sup\limits_{|z-\tilde{z}|\leq s_0, |z|\leq s_0}\frac{|{k(t,x,z)-k(t,x,\tilde{z})}|}{|z-\tilde{z}|^\nu}\leq 
\tau, \mbox{ for } x\in\mathbb{R}^N.
\end{equation} 


We present the definition of 
the $weak$-$(*)$-kernel condition,
which
is slightly weaker than 
the above $(*)$-kernel condition in Definition
\ref{def_parameter}.
\end{defn}
\begin{defn}\label{def_parameter2} 
Under the same setting of the parameters in Definition \ref{def_parameter}, 
we say that
 $K$
satisfies 
the $weak$-$(*)$-kernel condition 
on $[0,T]$ if 
$K$ 
satisfies \eqref{cond_symm} and \eqref{cond_bounds}
for the case $\alpha<1$.
If $\alpha\geq1$, we ask $K$ to
hold the following condition \eqref{cond_cancell} as well as
\eqref{cond_symm} and \eqref{cond_bounds}.
\begin{equation}\label{cond_cancell}
\odot
\mbox{ Cancellation: }
\Big| \int_{S^{N-1}} 
k(t,x,s\sigma)\sigma d\sigma\Big|\leq \tau \cdot s^\nu 
\mbox{ for  }
 s\in(0,s_0)
  \mbox{ and for   } x\in\mathbb{R}^N
\end{equation} 
where $\sigma$ is a surface element on the unit 
sphere
$S^{N-1}\subset
\mathbb{R}^N$.
\end{defn}
\begin{rem}\label{rem_really_weaker}
 The  $(*)$-kernel condition in Definition \ref{def_parameter} 
 implies
  the $weak$-$(*)$-kernel condition
  in Definition \ref{def_parameter2}. Indeed,
for the case $\alpha<1$, 
they
are exactly same. If $\alpha\geq1$, then
the only difference between them
  is
that 
the $(*)$-kernel condition needs \eqref{cond_Holder}
while the $weak$-$(*)$-kernel condition requires
\eqref{cond_cancell}.
Also, it is easy to verify that 
\eqref{cond_Holder} implies 
 \eqref{cond_cancell} up to a constant
 by the following argument:
for any $s\in(0,s_0/2)$,
\begin{equation*}\begin{split}
\Big| \int_{S^{N-1}} 
k(t,x,s\sigma)\sigma d\sigma\Big|&=
\Big| \int_{S^{N-1}_+} 
k(t,x,s\sigma)\sigma d\sigma
+\int_{S^{N-1}_-} 
k(t,x,s\sigma)\sigma d\sigma\Big|\\
&\leq
\int_{S^{N-1}_+}\Big|  
k(t,x,s\sigma)-k(t,x,-s\sigma)\Big| d\sigma.
\end{split}\end{equation*}  
where $S^{N-1}_+$ and $S^{N-1}_-$ are the upper and the lower
hemispheres, respectively. Then, thanks to
\eqref{cond_Holder}, we have
\begin{equation*}\begin{split}
&\leq
\int_{S^{N-1}_+}\tau\Big|2s\sigma \Big|^\nu 
 d\sigma\leq (C\tau)  \cdot s^\nu.
\end{split}\end{equation*} 

\end{rem}
\begin{rem}
In the work of \cite{ccv}, the upper bound for $k$ is just $\Lambda$ while,
in this paper, we have $\Lambda\cdot(1+ |x-y|^\omega)$ 
in \eqref{cond_bounds}, which is slightly more general than that of \cite{ccv}.
Some examples with the  upper bound \eqref{cond_bounds}
can be found in Section 4 of Komatsu \cite{komatsu_fundamental}.
\end{rem}
\begin{rem}
The purpose of the condition \eqref{cond_cancell}
with $(\alpha-1)<\nu$ is 
to consider $T_t^K(f)(\cdot)$ not only
as a distribution but also as a locally 
integrable function.
In general, without such an additional cancellation condition, if $\alpha\geq1$,
then $T_t^K(f)(x)$ is not well-defined even for $f\in C_c^\infty$.
In Lemma \ref{lem_property_kernel}
and Lemma
\ref{lem_fractional_laplacian}, it will be shown  that 
as long as the corresponding kernel $K$ satisfies the $weak$-$(*)$-kernel condition,
the operator $T^K_t$
is well defined, and 
$T^K_t(f)$ is a locally integrable function for
some class of functions $f$.

\end{rem}
\begin{rem}
Let the
kernel $K$ satisfy the $weak$-$(*)$-kernel condition
 for some $\alpha\geq1$. Then
we can combine the two conditions  \eqref{cond_bounds} and \eqref{cond_cancell}
in order to get an estimate of the integral in \eqref{cond_cancell} for all $s\in(0,\infty)$.
Indeed, the  condition \eqref{cond_bounds}
 implies that, for $s\in[s_0,\infty)$,
\begin{equation*}
 \int_{S^{N-1}}\Big| 
k(t,x,s\sigma) \Big|d\sigma\leq 
C\Lambda\cdot(1+s^\omega)\leq
(C\Lambda\cdot s_0^{-\nu} )\cdot s^\nu\cdot(1+s^\omega).
\end{equation*} 
Thus, together with the condition \eqref{cond_cancell},
we have, for $s\in(0,\infty)$,
\begin{equation}\label{cond_candcell_all_s}
\Big| \int_{S^{N-1}} 
k(t,x,s\sigma)\sigma d\sigma\Big|\leq \bar{\tau}
\cdot s^\nu\cdot(1+s^\omega)
\end{equation} where $\bar{\tau}:=\max\{\tau,(C\Lambda\cdot s_0^{-\nu})\}$.\\
\end{rem}
\begin{rem}\label{examples}
 We present some typical examples satisfying  either
the $(*)$-kernel condition or
the $weak$-$(*)$-kernel condition.\\
 
 
  (I) For the  simplest example,
if $K:=c_\alpha/|x-y|^{N+\alpha}$ (i.e. $k:=c_\alpha$), then the equation \eqref{main_lin_eq} becomes
the fractional heat equation (some regularity
results can be found in Caffarelli and   Figalli \cite{caffarelli_figalli}).
This kernel satisfies the  $(*)$-kernel condition.
Indeed,
\eqref{cond_symm} and \eqref{cond_bounds} are trivial.
For $\alpha\geq1$, since $k$ is a constant function,
\eqref{cond_Holder}
 holds for
 any $\nu\in(\alpha-1,1)$
  with $\tau=0$ and $s_0=\infty$.\\   
   
  (II) One may assume that the kernel has the form
  not of $K(t,x,y)$
  but of $K(t,x-y)$ (for more general cases, we refer to Silvestre \cite{silve:fractional_like}). Then the natural symmetry we would impose to the kernel is
  $K(t,x-y)=K(t,y-x)$,
  which implies \eqref{cond_symm} directly. This $K$ holds 
  the $weak$-$(*)$-kernel condition 
  for any $\alpha\in(0,2)$ once 
  we assume the bounds condition \eqref{cond_bounds}.
Indeed,   for $\alpha\geq1$,
   the integral in \eqref{cond_cancell} 
  is  always zero
  due to the cancellation
    from
   the symmetry
 (i.e. 
 any $\nu\in(\alpha-1,1)$
  with $\tau=0$ and $s_0=\infty$ works).\\  
\end{rem}

Here is our main theorem about persistence of H\"{o}lder continuity.
\begin{thm}\label{main_thm}
Let 
$\{{\alpha},\zeta,\omega,
\Lambda\}$ 
(for the case $\alpha\geq1$, 
$\{{\alpha},\zeta,\omega,
\Lambda,\nu,s_0,\tau\}$)
be a set of the parameters in Definition \ref{def_parameter}. 
Then there exist two constants  $\beta>0$ and $C>0$ with the following two properties $(I)$ and $(II)$:\\

$(I)$. Let  ${w}_0\in (L^1\cap L^\infty)(\mathbb{R}^N)$ be a given function
and let $0<T\leq\infty$.
Let $K$ satisfy the $(*)$-kernel condition
on $[0,T]$
(see Definition \ref{def_parameter}).
Then there exists a   weak solution $w$
of \eqref{main_lin_eq} on $(0,T)$
satisfying the following estimates for a.e.  $t\in(0,T)$:
\begin{equation}\label{eq_C_beta_main_thm}
      \|{w}(t,\cdot)\|_{C^\beta(\mathbb{R}^N)}
\leq C 
\cdot\|{w}_0\|_{C^\beta(\mathbb{R}^N)}
\quad\mbox{ if } w_0\in C^\beta(\mathbb{R}^N), 
     \end{equation}
\begin{equation}\label{eq_L_infty_main_thm}
      \|{w}(t,\cdot)\|_{C^\beta(\mathbb{R}^N)}
\leq C \cdot{\max\{1,\frac{1}{t^{\beta/\alpha}}\}}\cdot\|{w}_0\|_{L^\infty(\mathbb{R}^N)}, \quad\mbox{ and}
     \end{equation}
\begin{equation}\label{eq_L_1_main_thm}
      \|{w}(t,\cdot)\|_{C^\beta(\mathbb{R}^N)}
\leq C \Big(
\|{w}_0\|_{L^\infty(\mathbb{R}^N)}+
{\max\{1,\frac{1}{t^{(N+\beta)/\alpha}}\}}\cdot\|{w}_0\|_{L^1(\mathbb{R}^N)}\Big).\\
     \end{equation}
     
$(II)$. Let  ${w}_0\in C^\infty(\overline{\mathbb{R}^N})$
 be a given  function
and let $0<T\leq\infty$. 
Let $K$ satisfy the $weak$-$(*)$-kernel condition
on $[0,T]$
(see Definition \ref{def_parameter2}). In addition, we assume
\begin{equation}\begin{split}\label{smooth_assump_lin}
 &k(\cdot,\cdot,\cdot)
\in C^\infty_{t,x,z}\Big([0,T]\times\overline{\mathbb{R}^N}\times\overline{\mathbb{R}^N}\Big).
\end{split}\end{equation}
Suppose that
  ${w}\in L^\infty(0,T;L^2(\mathbb{R}^N))$ is a smooth solution of \eqref{main_lin_eq} 
  on $[0,T)\times\mathbb{R}^N$
for the initial data 
${w}_0$. 
Then, $w$ satisfy
 \eqref{eq_C_beta_main_thm},
 \eqref{eq_L_infty_main_thm}, and 
 \eqref{eq_L_1_main_thm}  for any $t\in(0,T)$.

\end{thm}


We concentrate our effort first to prove
the part $(II)$ of the above theorem in Section
\ref{sec_Preliminaries},
\ref{sec_Proof_of_main_prop}, and
\ref{sec_Proof_of_main_thm}. 
In fact, we will show the part $(II)$ carefully
to ensure that
  the two constants $C$ and $\beta$ 
  in the  conclusion of the part $(II)$
depend only on the parameters in Definition
\ref{def_parameter}. Thus,
      these  two constants $C$ and $\beta$
      depend neither on
 $T$ nor on any actual norms coming from 
the smoothness assumption \eqref{smooth_assump_lin}.
As a result,
 the part $(I)$, which will be proved
in Appendix, 
follows
 the part $(II)$ by a limit argument.
 Unfortunately, if $\alpha\geq1$, then we need
 the condition \eqref{cond_Holder},
 which is more restrictive than
 \eqref{cond_cancell}. \\

\begin{rem} More precisely, 
the conclusion of the part $(I)$ 
follows once we regularize the function $k$ in a proper way, which should keep all the parameters.
In short, since $k$ may not be bounded due to \eqref{cond_bounds},
we make it bounded first. Then take a convolution with a mollifier.
This process does not hurt the parameter set essentially if $\alpha<1$.
However for the case $\alpha\geq1$, the cancellation condition
\eqref{cond_cancell} is not preserved during the process. That is
the reason we 
impose 
the  $(*)$-kernel condition
to the part $(I)$ of Theorem \ref{main_thm} 
instead of 
the $weak$-$(*)$-kernel condition.
\end{rem}
\begin{rem}
Thanks to the symmetry condition \eqref{cond_symm}, we use the following
 weak formulation of \eqref{main_lin_eq}:
 \begin{equation*}
 \int_{\mathbb{R}^N}(\partial_tw)(t,x)\eta(x) dx
 +\frac{1}{2}\iint_{\mathbb{R}^N\times\mathbb{R}^N}
 [w(t,x)-w(t,y)][\eta(x)-\eta(y)]K(t,x,y)dydx=0
\end{equation*} for $\eta(\cdot)\in C_c^\infty(\mathbb{R}^N)$ and for a.e. $t$
(e.g.  see \cite{ccv} or \cite{kassman_apriori}).\\
\end{rem}

As in \cite{ccv}, we
show how our result can be applied to a fully 
non-linear problem.
We introduce 
the following non-linear evolution problem:
\begin{equation}\label{main_non_lin_eq}
 \partial_t \theta(t,x)-
\int_{\mathbb{R}^N}\phi^\prime(\theta(t,y)-\theta(t,x))G(y-x)dy=0.
\end{equation} This equation can be considered as
the evolution problem coming from the Euler-Lagrange equation
for the variational integral 
\begin{equation*}
\iint_{\mathbb{R}^N\times\mathbb{R}^N}\phi(\theta(t,y)-\theta(t,x))G(y-x)dydx
\end{equation*} (for more
detailed explanation, see \cite{ccv}). This   non-linear problem 
can be found in 
Giacomin, Lebowitz,  and Presutti 
\cite{MR1661764}, or in the field of 
image processing (e.g. see 
Gilboa and Osher \cite{osher},
Lou,   Zhang,  Osher,  and Bertozzi \cite{MR2578033}).\\

We impose the following conditions to the equation \eqref{main_non_lin_eq}.\\

Let $0<\alpha<2$,  $0<\zeta\leq\infty$,
$0\leq\omega<{\alpha}$ and 
$1\leq\Lambda<\infty$
(for the case $\alpha\geq1$, we need two more parameters $\nu$  and $M$
such that $(\alpha-1)<\nu<1$, $0\leq M<\infty$, and $\nu+\omega<\min\{N,\alpha\}$).
Let  $\phi:\mathbb{R}\rightarrow [0,\infty)$ be
 an even function of class $C^2$ 
satisfying \begin{equation*}
            \phi(0)=0 \quad\mbox{ and }\quad \sqrt{{\Lambda^{-1}}}
\leq \phi^{\prime\prime}(x)\leq  \sqrt{\Lambda}, \quad x\in\mathbb{R}.
           \end{equation*}
We assume that the kernel $G:\mathbb{R}^N/\{0\}\rightarrow[0,\infty)$ satisfies
$G(-x)=G(x)$ and 
\begin{equation}\label{cond_bounds_kernel_nonlinear}
\sqrt{{\Lambda^{-1}}}\cdot\mathbf{1}_{|x|\leq \zeta}\leq G(x)\cdot|x|^{N+{\alpha}}
\leq \sqrt{\Lambda}\cdot(1+|x|^\omega) \mbox{ for   } x\in\mathbb{R}^N/\{0\}.
\end{equation}
Let  ${\theta}_0\in (W^{1,1}\cap
W^{1,\infty}
 )(\mathbb{R}^N)$ be a
 given
   function. 
For the case $\alpha\geq1$, we assume further
\begin{equation*}
  \phi^{\prime\prime}\in C^{\nu}(\mathbb{R}) \mbox{ with }
[\phi^{\prime\prime}]_{C^{\nu}(\mathbb{R})}\cdot
 \|\nabla{\theta}_0\|^{\nu}_{L^\infty(\mathbb{R}^N)} \leq M.
\end{equation*}

\begin{rem}
The upper bound \eqref{cond_bounds_kernel_nonlinear} for $G(\cdot)$  is 
$\sqrt{\Lambda}\cdot(1+|x|^\omega)$ and
it is more flexible than that of  \cite{ccv},
where just $\sqrt{\Lambda}$ was used as the upper bound.\\
\end{rem}

Following the approach of \cite{ccv}, we present
 the following important consequence of 
 the part $(II)$ of Theorem 
 \ref{main_thm}.  
 \begin{thm}\label{main_thm_nonlinear}

We have two constants $\beta>0$ and $C>0$
which depend only on the above parameters, 
and there exists a global-time weak solution  ${\theta}
$
of the equation \eqref{main_non_lin_eq}
with the following estimates for  a.e. $t\in(0,\infty)$:
\begin{equation}\label{c_1,beta}
      \|\nabla{\theta}(t,\cdot)\|_{C^\beta(\mathbb{R}^N)}
\leq C 
\cdot\|\nabla{\theta}_0\|_{C^\beta(\mathbb{R}^N)} \quad\mbox{ if }\quad\nabla\theta_0\in {C^\beta(\mathbb{R}^N)}  , 
     \end{equation}
\begin{equation*}
      \|\nabla{\theta}(t,\cdot)\|_{C^\beta(\mathbb{R}^N)}
\leq C \cdot{\max\{1,\frac{1}{t^{\beta/\alpha}}\}}\cdot
\|\nabla{\theta}_0\|_{L^\infty(\mathbb{R}^N)}, \quad\mbox{ and}
     \end{equation*}
\begin{equation*}
      \|\nabla{\theta}(t,\cdot)\|_{C^\beta(\mathbb{R}^N)}
\leq C \Big(
\|\nabla{\theta}_0\|_{L^\infty(\mathbb{R}^N)}+
{\max\{1,\frac{1}{t^{(N+\beta)/\alpha}}\}}\cdot\|\nabla{\theta}_0\|_{L^1(\mathbb{R}^N)}\Big).
     \end{equation*}

\end{thm}

The main idea of the above theorem \ref{main_thm_nonlinear} is the following:
First, we regularize $\theta_0$, $G$, and
$\phi$ in a proper way so that we obtain
a sequence of smooth solutions of \eqref{main_non_lin_eq}. Then, 
 we take a derivative $(w:=D_e\theta)$ to the non-linear equation
\eqref{main_non_lin_eq},
and  we freeze some coefficients.
 As a result of this process,
we obtain the linear equation \eqref{main_lin_eq}
together with the $weak$-$(*)$-kernel condition
on the $K$
satisfying \eqref{smooth_assump_lin}. Thus, we can use the conclusion
of the part $(II)$ of Theorem \ref{main_thm}.
Finally, we extract a weak solution by a limit 
argument.
This
 proof will be given in Appendix.\\

Now we want to explain the main idea 
 of the part $(II)$ of
 Theorem \ref{main_thm}, which is the heart of this paper.
As mentioned earlier, our proof 
 follows the spirit of the paper \cite{variation_kiselev}. 
 First, thanks to the duality of the equation \eqref{main_lin_eq}
 from the symmetry condition \eqref{cond_symm}, we can focus only on  
 the evolution of  $\mathcal{U}_r$,
 a class of test functions,
 which is related to the dual space of the H\"{o}lder space $C^\beta$ (see
 Definition \ref{def_u_r} of $\mathcal{U}_r$, Lemma
\ref{lem_c_beta}, and Lemma \ref{lem_dual}).
In this paper, we take the same definition of  the class $\mathcal{U}_r$
from the paper \cite{variation_kiselev}
 while
other classes can be found in Dabkowski
 \cite{eventual_dabkowski} and
 Chamorro \cite{chamorro}. In particular,
the class introduced in \cite{eventual_dabkowski} is quite different from that of 
\cite{variation_kiselev} and it was successfully used to obtain
eventual regularity of the super-critical surface quasi-geostrophic (SQG) equation.\\

Second, we prove the short-time evolution of 
test functions (Proposition \ref{main_prop}). In order to obtain
it, we need to manage the competition
(refer to 
Remark \ref{rem_competition}) between
the $L^p$ condition and the concentration condition.
The former condition, which can be proved from the lower bound 
${\Lambda^{-1}}\cdot\mathbf{1}_{|z|\leq \zeta}$,
of the kernel 
has a regularization effect 
(Lemma \ref{lem_L_infty}, Lemma \ref{lem_L_1})  as
 a diffusion term in usual PDEs does. However, the latter condition  comes from the upper bound 
$\Lambda\cdot(1+ |z|^\omega)$ 
of the kernel and this upper bound
plays a similar role as a source term in usual PDEs  
(Lemma \ref{lem_concent}).\\

In addition, since the length of the time interval
coming from the conclusion of Proposition
\ref{main_prop} is proportional to
 $r^\alpha$ where $r$ is the parameter of $\mathcal{U}_r$,
 it should be verified that
we can repeat the short-time evolution 
(Proposition \ref{main_prop}) as many times as we want
in order to reach any fixed time
(refer to Remark \ref{rem_repeat}).\\

For the case $\alpha<1$, the main difficulty
is to handle both lower and upper bounds 
\eqref{cond_bounds} of the kernel:
in particular, both the finite size $\zeta$ of support of the lower bound
and the term $(1+|z|^\omega)$ of the upper bound cause some troubles.
In order to cover the case $\alpha\geq1$,
we use  the cancellation condition \eqref{cond_cancell},
which is designed to cancel desirable amount of singularity at $x=y$ of the kernel. 
Then we can  interpret $T_t^K(f)$ as locally integrable
functions for some class of functions (see Lemma \ref{lem_fractional_laplacian},
  Lemma \ref{lem_property_kernel}).
This fact will be crucial to prove 
the concentration condition (Lemma \ref{lem_concent}).\\

We want to mention a few articles related to
the integral operator $T_t^K$ corresponding a kernel $K$.   
    For smooth bounded kernels, we may use
a theory of pseudo differential operators (e.g. Kumano-go \cite{Kumano-go},
Komatsu
 \cite{komatsu_continuity}), while
for measurable kernels, there exists a fundamental solution 
(see  \cite{komatsu_fundamental}). 
Also, we 
refer to
\cite{kassman_apriori} and 
Barlow,     Bass,  Chen,   and
              Kassmann \cite{BBCK}. Recently, in  Dyda and  Kassmann \cite{Dyda_kassmann},
assumptions of kernels have been extended in some geometrical sense.
If we focus on non-divergence case, we refer to
 \cite{caffa_silve_regular}. \\

As mentioned before, the following three
sections  \ref{sec_Preliminaries},
\ref{sec_Proof_of_main_prop}, and
\ref{sec_Proof_of_main_thm}
are dedicated to the proof of 
the part $(II)$ of the main theorem \ref{main_thm}. More precisely,
in Section \ref{sec_Preliminaries}, we 
introduce some definitions and few important lemmas.
After that, we present and prove the main proposition \ref{main_prop}
in Section \ref{sec_Proof_of_main_prop}. 
Finally, the proof of the part $(II)$ of  Theorem \ref{main_thm} ends
in Section \ref{sec_Proof_of_main_thm}.
At the end of this paper,
 Appendix contains the proofs of 
 the part $(I)$ of
 Theorem \ref{main_thm} and Theorem \ref{main_thm_nonlinear}.

\section{Preliminaries and lemmas 
}\label{sec_Preliminaries}

From now on, we fix  
a parameter set $\{{\alpha},\zeta,\omega,\Lambda\}$, 
which appears in Definition \ref{def_parameter}
(for the case $\alpha\geq1$, 
 $\{{\alpha},\zeta,\omega,\Lambda,\nu,s_0,\tau\}$).
Also, suppose that $K$ satisfies the $weak$-$(*)$-kernel condition  in Definition \ref{def_parameter2}
on the parameter set 
together with the
smoothness assumption \eqref{smooth_assump_lin}.
For the case $\alpha<1$, we define and fix a constant $\gamma$ such that 
$0<\gamma<(\alpha-\omega)$ 
 while, for the case
$\alpha\geq1$, we take $\gamma$ to be 
$\max\{(\alpha-N),0\}<\gamma<
(\alpha-(\omega+\nu))$).\\

Before considering  a general $\zeta\in(0,\infty]$,
we will prove first the conclusion of
the part $(II)$ of  Theorem \ref{main_thm} for 
a fixed $\zeta=\zeta_0$ where  
\begin{equation}\label{def_zeta_0}
\zeta_0
:=\max\{\Big(
\frac{8}{V_N}\Big)^{1/N}, 2\cdot(11)^{1/\gamma}\}
\end{equation}
($V_N$ is the volume of  the unit ball 
 in $\mathbb{R}^N$). 
This definition of $\zeta_0$ will help us to obtain enough regularization directly so that
the proof becomes more straightforward.
 Once we prove the part $(II)$ of Theorem \ref{main_thm} with $\zeta=\zeta_0$, a general proof
for any value ${\zeta}\in(0,\infty]$ will follow a scaling argument.
Indeed, the case $\zeta>\zeta_0$ is included in the case $\zeta=\zeta_0$ because
$\{|x-y|\leq\zeta_0\}\subset\{|x-y|\leq\zeta\}$. On the other hand,
for the case $\zeta<\zeta_0$, we define a scaling:
${w}^\epsilon(t,x)={w}(\epsilon^\alpha t,
\epsilon x)$  and $K^\epsilon(t,x,y)=K(\epsilon^\alpha t,
\epsilon x,\epsilon y)$. Thus, 
if ${w}$ satisfies \eqref{main_lin_eq}  on $[0,T]$ for a kernel $K$ with $\zeta<\zeta_0$, then ${w}^\epsilon$ is a solution on $[0,T/\epsilon^\alpha ]$
for the kernel $K^\epsilon$ with a new $\zeta=\zeta_0$
once we pick up  $\epsilon$ by $\epsilon=\zeta/ \zeta_0$. Then
we can apply the part $(II)$ of Theorem \ref{main_thm} for ${w}^\epsilon$ and the same result for ${w}$ follows.\\


In this paper, we denote Sobolev spaces by $W^{k,p}$ and 
$H^k:=W^{k,2}$ for integers $k\geq0$ and
for $p\in[1,\infty]$
in the usual way. In addition,
 the symbol $\mathcal{S}$ is used to represent the Schwartz space in $\mathbb{R}^N$.


\begin{defn}
We say that a  function $f$ lies in $C^k(\mathbb{R}^d)$ for an integer $k\geq 0$
if $f$ is $k$-times differentiable in $\mathbb{R}^d$ and
all derivatives up to $k$ order are continuous,
while $f$ lies in $C^k(\overline{\mathbb{R}^d})$ if $f\in C^k(\mathbb{R}^d)$ and if $\nabla^l f $ are bounded 
for all integer $l$ such that $0\leq l \leq k$. In other words, $ C^k(\overline{\mathbb{R}^d})=C^k(\mathbb{R}^d)\cap W^{k,\infty}(\mathbb{R}^d).$

\end{defn}
\begin{defn}
We say that a bounded function $f$ lies in 
$C^\beta(\mathbb{R}^d)$ for $0<\beta<1$
if $\sup_{x,y}{|f(x)-f(y)|}/{|x-y|^\beta}$ is finite and
 we define the semi-norm
  $[f]_{C^\beta}:=\sup_{x,y}{|f(x)-f(y)|}/{|x-y|^\beta}$
and the norm $
\|f\|_{C^\beta}:=\|f\|_{L^\infty}+
[f]_{C^\beta}$. We also define the space $C^{k,\beta}(\mathbb{R}^d)$
by the norm $
\|f\|_{C^{k,\beta}}:=\|f\|_{W^{k,\infty}}+
[\nabla^k f]_{C^\beta}$
\end{defn}

It will be shown 
in Lemma \ref{lem_property_kernel} that the operator $T^K_t(f)$ is well-defined pointwise
for $f\in (C^2\cap L^1)(\mathbb{R}^N)$.
 Moreover
the operator  can be extended  to more general spaces.
For example, if $f$ is locally integrable and
$\int \frac{|f|}{1+|x|^{N+\alpha-\omega}}dx<\infty$,
 then we can define $T^K_t(f)$ as an element
of $\mathcal{S}^\prime$ where 
$\mathcal{S}^\prime$ is 
the dual of 
Schwartz space $\mathcal{S}$ (see also
Silvestre 
\cite{silve:fractional}).
We will make use of
the following Lemma \ref{lem_fractional_laplacian}, which  says that
$T^K_t(|\cdot|^\gamma)$ is not only an element of $\mathcal{S}^\prime$
but also a locally integrable function with a desirable estimate. This fact 
will be used to obtain the concentration 
condition (Lemma
\ref{lem_concent}) for the evolution
of $\mathcal{U}_r$, which will be introduced in
Definition \ref{def_u_r}.
\begin{lem}\label{lem_fractional_laplacian}
We have an estimate
\begin{equation*}
\Big|T^K_t\Big(|\cdot|^\gamma\Big)(x)\Big|\leq 
\begin{cases}& C\cdot
 |x|^{{\gamma}-{\alpha}}\cdot(1+|x|^\omega),
\quad\mbox{ if } 0<\alpha<1,\\
 & C\cdot
  |x|^{{\gamma}-{\alpha}}\cdot
 (1+|x|^{{\nu}+\omega}),
  \quad\mbox{ if } 1\leq\alpha<2.\\
\end{cases}
\end{equation*} 

\end{lem}
\begin{rem}
Recall that $\gamma$ is a fixed constant such that $0<\gamma<\alpha-\omega$  (for $\alpha\geq1$,
$\max\{(\alpha-N),0\}<\gamma<
(\alpha-(\omega+\nu))$). 
\end{rem}
\begin{proof}
 \begin{equation*}
  \begin{split}
  &\int_{\mathbb{R}^N}\Big[|x|^\gamma-|y|^\gamma\Big]K(t,x,y)dy=
\int_{\mathbb{R}^N}\Big[|x|^\gamma-|x+z|^\gamma\Big]K(t,x,x+z)dz\\
&=\int_{\mathbb{R}^N}\frac{|x|^\gamma-|x+z|^\gamma}{|z|^{N+\alpha}}k(t,x,z)dz\\
&=|x|^\gamma \int_{\mathbb{R}^N}\frac{1-|\frac{x}{|x|}+\frac{z}{|x|}|^\gamma}{|z|^{N+\alpha}}k(t,x,z)dz.
 \end{split}
 \end{equation*} Then we use the change of
 variables $\frac{z}{|x|}=\bar{z}$
 and the polar coordinate
 to get
 \begin{equation*}
  \begin{split}
&=|x|^{\gamma-\alpha} \int_{\mathbb{R}^N}\frac{1-|\frac{x}{|x|}
+\bar{z}|^\gamma}{|\bar{z}|^{N+\alpha}}k(t,x,|x|\bar{z})d\bar{z}\\
&=|x|^{\gamma-\alpha} \int_{S^{N-1}}\int_{0}^{\infty}\frac{1-|\frac{x}{|x|}
+s\sigma|^\gamma}{s^{1+\alpha}}k(t,x,|x|s\sigma)dsd\sigma\\
&=|x|^{\gamma-\alpha}\Big( \int_{S^{N-1}}\int_{0}^{1/2}\cdots dsd\sigma
+ \int_{S^{N-1}}\int_{1/2}^{\infty}\cdots dsd\sigma\Big)=|x|^{\gamma-\alpha}
\Big((I)+(II)\Big).\\
 \end{split}
 \end{equation*}
From the condition $\gamma+\omega<\alpha$, we have
 \begin{equation*}
  \begin{split}
\Big|(II)\Big|&\leq
\Lambda\int_{S^{N-1}}\int_{1/2}^{\infty}
\frac{(2+s^\gamma)(1+|x|^\omega s^\omega)}{s^{1+\alpha}}dsd\sigma\leq C
(1+|x|^\omega). 
 \end{split}
 \end{equation*}
On the other hand, from  Taylor expansion $1-|\frac{x}{|x|}
+s\sigma|^\gamma = -\gamma(\frac{x}{|x|}\cdot\sigma)s+R_\sigma(s)$
with an error estimate $|R_\sigma(s)|\leq Cs^2$ 
for $s\in[0,1/2]$ and $\sigma\in S^{N-1}$,
we have 
\begin{equation*}
  \begin{split}
\Big|(I)&\Big|\leq 
\Big|\int_{S^{N-1}}\int_{0}^{1/2}
\gamma(\frac{x}{|x|}\cdot\sigma)\frac{s}{s^{1+\alpha}} 
k(t,x,|x|s\sigma) dsd\sigma\Big|\\
&\quad\quad\quad\quad\quad
+\Lambda\int_{S^{N-1}}\int_{0}^{1/2}
\frac{|R_\sigma(s)|(1+|x|^\omega s^\omega)}{s^{1+\alpha}}
 dsd\sigma\\
&\leq C \int_{0}^{1/2}
\frac{1}{s^{\alpha}}
\Big|\int_{S^{N-1}}
k(t,x,|x|s\sigma) \sigma d\sigma \Big| ds
+C\cdot\frac{1+|x|^\omega}{2-\alpha}\cdot\Lambda\\
&=C\cdot(III)+C\cdot(1+|x|^\omega).
 \end{split}
 \end{equation*}
For the case $\alpha<1$, $(III)$ is bounded above by $C\cdot
(1+|x|^\omega)$.\\
If $\alpha\geq 1$, we can use 
the condition 
\eqref{cond_candcell_all_s}, which is obtained from
  \eqref{cond_cancell}
and \eqref{cond_bounds},
together with $\nu>(\alpha-1)$:
\begin{equation*}
  \begin{split}
(III)&\leq
\int_{0}^{1/2}
\frac{1}{s^{\alpha}}
\bar{\tau} |x|^\nu s^\nu(1+|x|^\omega s^\omega) ds\leq
\bar{\tau} |x|^\nu \cdot
(1+|x|^\omega)
\int_{0}^{1/2}
\frac{1}{s^{\alpha-\nu}}
 ds\\
 &
\leq
C\cdot
(1+|x|^\omega)
 \cdot|x|^\nu .
 \end{split}
 \end{equation*}

\end{proof}

We give, in the following lemma, some properties of the integral operator $T^K_t$ and the related
evolution equation \eqref{main_lin_eq}.
\begin{lem}\label{lem_property_kernel}
For any $f\in C^2({\mathbb{R}^N})\cap L^1(\mathbb{R}^N)$, 
$T^K_t(f)$ is well-defined pointwise. Moreover,
the following properties hold:\\ 
(I).  Duality of T:
\begin{equation}\label{duality_of_T}
 \int f(x)T^K_t(g)(x)dx=
\int T^K_t(f)(x)g(x)dx,
\end{equation} for $f\in C^2(\overline{\mathbb{R}^N})\cap L^1(\mathbb{R}^N)$ and either
$g\in C^2(\overline{\mathbb{R}^N})\cap L^1(\mathbb{R}^N)$
or $g(x)=|x|^\gamma$.\\
(II). Mean zero of T:
\begin{equation*}
 \int T^K_t(f)(x)dx=0,
\end{equation*} for $f\in C^2(\overline{\mathbb{R}^N})\cap L^1(\mathbb{R}^N)$.
\end{lem}
\begin{proof}
%
Let $f\in C^2({\mathbb{R}^N})\cap L^1(\mathbb{R}^N)$. Then, we have 
 \begin{equation*}\begin{split}
    &T^K_t(f)(x)
    =\int(f(x)-f(y))K(t,x,y)dy
    =\int\frac{f(x)-f(x+z)}{|z|^{N+\alpha}}k(t,x,z)dz\\
    &=\int_0^1\int_{S^{N-1}}\frac{(\nabla f)(x)\cdot(r\sigma)+R_f(x,r\sigma)}{r^{1+\alpha}}
        k(t,x,r\sigma)d\sigma dr\\
        &\quad\quad\quad\quad\quad+\int_{|z|\geq1}\frac{
f(x)-f(x+z)
        }{z^{N+\alpha}}
                k(t,x,z)dz=(a)+(b).
                     \end{split}\end{equation*}
                     where we used the Taylor expansion of $f$ in the first integral.\\
                     For $(b)$, we use the upper bound 
 of                                         \eqref{cond_bounds}:
                     \begin{equation*}\begin{split}
                     |(b)|&
\leq \Lambda\int_{|z|\geq1} \Big(|f(x)|+|f(x+z)|\Big)\frac{1+|z|^\omega}{|z|^{N+\alpha}}dz\\
&\leq C|f(x)| \int_{|z|\geq1}
\frac{1}{|z|^{N+\alpha-\omega}}dz   +  C\int_{|z|\geq1}|f(x+z)|dz\\
& \leq C|f(x)| + \|f\|_{L^1}.               
                     \end{split}\end{equation*}
                     For $(a)$, if $\alpha<1$, we use  $|R_f(x,r\sigma)|\leq C\cdot \|\nabla^2 f\|_{L^\infty(B_x(1))
                                                               }\cdot r^2$ from the Taylor error estimate 
                                                               where
                                                               $B_x(r)$ is the ball of  radius $r$ centered at $x$:
 \begin{equation*}\begin{split}
                     |(a)|&\leq C|\nabla f(x)|\cdot\int_0^1 \frac{r}{r^{1+\alpha}}
       +C\|\nabla^2 f\|_{L^\infty(B_x(1))}\cdot\int_0^1 \frac{r^2}{r^{1+\alpha}}\\
       &\leq C
       \Big(|\nabla f(x)|+\|\nabla^2 f\|_{L^\infty(B_x(1))}\Big).
                      \end{split}\end{equation*}
                       
                      If $\alpha\geq1$, we use the condition \eqref{cond_candcell_all_s} with the assumption $(\alpha-1)<\nu$:
                                     \begin{equation*}\begin{split}                      
                        (a)&\leq C|\nabla f(x)|\cdot\int_0^1
                        \frac{1}{r^\alpha}\Big|\int_{S^{N-1}} k(t,x,r\sigma)\sigma d\sigma\Big|dr
                         +C\|\nabla^2 f\|_{L^\infty(B_x(1))}\\
                         &\leq C|\nabla f(x)|\cdot\int_0^1
                                                 \frac{1}{r^\alpha}
                                                 r^\nu
                                                 dr
                                                  +C\|\nabla^2 f\|_{L^\infty(B_x(1))}\\
                     &                             \leq C
                                                         \Big(|\nabla f(x)|+\|\nabla^2 f\|_{L^\infty(B_x(1))}\Big).
                                          \end{split}\end{equation*}
Now we can easily verify that $T^K_t(f)$                                          
is well-defined pointwise.\\

Note that                                      
        if  $f\in C^2(\overline{\mathbb{R}^N})\cap L^1(\mathbb{R}^N)$,
        then the above argument implies $T^K_t(f)\in L^\infty$. 
Then, the proof of $(I)$ follows the symmetry in $x,y$ of $K$. Indeed,
if  $f,g\in C^2(\overline{\mathbb{R}^N})\cap L^1(\mathbb{R}^N)$, 
then we have 
\begin{equation}\begin{split}\label{dual_computation}
   \int f(x)T^K_t(g)(x)dx=  &  \int \int f(x)(g(x)-g(y))K(t,x,y)dy  dx \\&=
\frac{1}{2}\int \int (f(x)-f(y))(g(x)-g(y))K(t,x,y)dy  dx \\    
&=
\int \int (f(x)-f(y))g(x)K(t,x,y)dy  dx \\
&=\int T^K_t(f)(x)g(x)dx.\\      
                    \end{split}\end{equation}
In addition, for the case   $f\in C^2(\overline{\mathbb{R}^N})\cap L^1(\mathbb{R}^N)$ with  $g(x)=|x|^\gamma$, then   
the 
 integral $\int |f(x)T^K_t(g)(x)|dx$
         is bounded due to the assumption $f\in L^\infty\cap L^1$ with Lemma  \ref{lem_fractional_laplacian}.
         Indeed, 
         Lemma  \ref{lem_fractional_laplacian}
         implies that $T^K_t(|\cdot|^\gamma)$ is  integrable
         in the unit ball containing the origin
         and is bounded outside of the ball. 
         Then, together with 
         $f\in L^\infty\cap L^1$, we obtain
         $f\cdot T^K_t(|\cdot|^\gamma)\in L^1$.
          Thus all equalities of \eqref{dual_computation} can be justified via a limit argument. \\

To prove $(II)$, we  take $g\in C_c^\infty$ such that $g=1$ in $B(1)$
and $supp(g)\subset B(2)$ and define 
$g_n$ by $g_n(\cdot):=g(\cdot/n)$. Then, thanks to the property  $(I)$ with $g_n$, the conclusion follows by taking a limit $n\rightarrow\infty$.
\end{proof}

In the following lemma, we present a
maximum principle for solutions of  \eqref{main_lin_eq}.
\begin{lem}\label{lem_property_equation_with_kernel}
Suppose that $w\in L^\infty_t(L^\infty_x\cap L^1_x)$ is a 
smooth
solution of \eqref{main_lin_eq}.
Then, the following properties hold:\\
(I). For any 
$C^2$ convex function $\eta:\mathbb{R}\rightarrow\mathbb{R}$,
we have 
\begin{equation}\label{ineq_convex}
\partial_t (\eta(w))+T^K_t(\eta(w))\leq 0.
\end{equation} 
(II). $L^p-$norm is non-increasing for $1\leq
p\leq\infty$:\\
\begin{equation}\label{L_P_maximum}
 \|w(t,\cdot)\|_{L^p(\mathbb{R}^N)}\leq
\|w(s,\cdot)\|_{L^p(\mathbb{R}^N)} \mbox{ for any } s<t.
\end{equation}
\end{lem}
\begin{rem}
Also we assume that the solutions are smooth.
However
 the estimate of the result does
not depend on this smoothness. 
\end{rem}
\begin{proof}
To prove $(I)$, we multiply $\eta^\prime(w)$ to the equation \eqref{main_lin_eq} to get
$\partial_t (\eta(w))+\eta^\prime(w)T^K_t(w)=0$. Then it is enough to show
$\eta^\prime(w)T^K_t(w)-T^K_t(\eta(w))\geq0$.
Using the integral representation of $T^K_t$, we have

\begin{equation*}\begin{split}
&\eta^\prime(w(x))\Big(T^K_t(w)\Big)(x)-\Big(T^K_t(\eta(w))\Big)(x)= \\
&\int(w(x)-w(y))K(x,y)\eta^\prime(w(x)) dy
-\int(\eta(w(x))-\eta(w(y)))K(x,y)dy\\
&=\int\Big(\eta^\prime(w(x)) (w(x)-w(y))-(\eta(w(x))-\eta(w(y)))\Big)
K(x,y)dy\geq 0\\
\end{split}\end{equation*} because 
$\eta^\prime(a) (a-b)-(\eta(a)-\eta(b))\geq0$ from convexity of $\eta$.\\

To prove $(II)$, 
 for $p\geq2$, we  use 
\eqref{ineq_convex} 
 with putting 
$\eta(\cdot)=
|\cdot|^p$ and taking an integral in $x$ variable
in order to use $(II)$ of Lemma \ref{lem_property_kernel}.
For $p<2$,  we need to regularize $\eta(\cdot)=
|\cdot|^p$ first. 
\end{proof}

Now we adopt the notion of the class $\mathcal{U}_r$
of test functions following the paper \cite{variation_kiselev}. Let $A\geq 1$ be a constant which will be chosen later.
\begin{defn}\label{def_u_r} We say that a 
measurable
 function $\varphi(\cdot)$ on $\mathbb{R}^N$ lies in 
$\mathcal{U}_r$ for some $r\in(0,\infty)$ if  $\varphi$ satisfies the following four conditions:
\begin{equation*}
 \int_{\mathbb{R}^N}\varphi(x)dx=0 \quad\quad \mbox{the mean zero-condition,}
\end{equation*}
\begin{equation*}
 \int_{\mathbb{R}^N}|\varphi(x)||x-x_0|^\gamma dx\leq r^\gamma 
\quad \mbox{for some }x_0\in\mathbb{R}^N \quad \mbox{the concentration-condition,}
\end{equation*}
\begin{equation*}
 \|\varphi\|_{L^\infty(\mathbb{R}^N)}\leq \frac{A}{r^N} \quad\quad \mbox{ the } L^\infty\mbox{-condition, and}
\end{equation*}
\begin{equation*}
 \|\varphi\|_{L^1(\mathbb{R}^N)}\leq 1\quad\quad \mbox{ the } L^1\mbox{-condition}.
\end{equation*} In addition, we say that $\varphi$ lies in $ a\mathcal{U}_r$ for some $a>0$ when  $(1/a)\varphi\in\mathcal{U}_r$. We call $x_0$ 
a center of $\varphi$.
\end{defn}

The following lemma connects between
$C^\beta$ space and $\mathcal{U}_r$,
which tells us that  $r^{-\beta}\mathcal{U}_r$
plays a similar role of the dual space of $C^\beta$.
\begin{lem}\label{lem_c_beta}
Let $\beta$ be any constant such that $0<\beta\leq\gamma$.\\
(I)  
Then  we have 
\begin{equation*}
 \Big|\int_{\mathbb{R}^N}{w}(x)\varphi(x)dx\Big|\leq  r^\beta [{w}]_{C^\beta(\mathbb{R}^N)}
\end{equation*}
 for any ${w}\in C^\beta(\mathbb{R}^N)$,
 for any $0<r<\infty$,   and for any $\varphi\in\mathcal{U}_r$. \\

(II) Conversely,
 we have a constant $C$ such that \\
if  a bounded function ${w}$ satisfies
$\mbox{ } \sup\limits_{\varphi\in
\mathcal{S}\cap
\mathcal{U}_r,0<r\leq 1}
r^{-\beta}\Big|\int_{\mathbb{R}^3}{w}(x)\varphi(x)dx\Big|<\infty$,\\
then ${w}\in C^\beta$ and
\begin{equation}\label{lem_c_beta_(II)}
\|{w}\|_{C^\beta(\mathbb{R}^N)}\leq C\cdot \Big(\|{w}\|_{L^\infty} + 
\sup\limits_{\varphi\in
\mathcal{S}\cap
\mathcal{U}_r,0<r\leq 1}
r^{-\beta}\Big|\int_{\mathbb{R}^3}{w}(x)\varphi(x)dx\Big| \Big).
\end{equation}
\end{lem}
\begin{proof}
For the part \textit{(I)}, let $x_0$ be a center of   $\varphi$. Then, from
the mean zero property,
\begin{equation*}\begin{split}
 &\Big|\int_{\mathbb{R}^N}{w}(x)\varphi(x)dx\Big|
\leq  |\int_{\mathbb{R}^N}\Big(
{w}(x)-{w}(x_0)\Big)\varphi(x)dx\Big| \\
&\leq [{w}]_{C^\beta(\mathbb{R}^N)}
\int_{\mathbb{R}^N}|
x-x_0|^{\beta}|\varphi(x)|dx \\
&\leq [{w}]_{C^\beta(\mathbb{R}^N)}
\Big(\int_{\mathbb{R}^N}|
x-x_0|^{\gamma}|\varphi(x)|dx \Big)^{\beta/\gamma}
\Big(\int_{\mathbb{R}^N}|\varphi(x)|dx \Big)^{(\gamma-\beta)/\gamma}
 \leq  r^\beta [{w}]_{C^\beta(\mathbb{R}^N)}.
\end{split}\end{equation*}
For the part \textit{(II)}, we recall 
Littlewood-Paley projections $\Delta_j$,
which is defined by
$\Delta_j({w})={w} * \Psi_{2^{-j}}$ where
$\Psi_{t}(x)=t^{-N}\Psi(x/t)$ and $\Hat{\Psi}(\xi)=\eta(\xi)-\eta(2\xi)$
with $\eta\in C^\infty_0, 0\leq \eta(\xi)\leq 1, \eta=1 \mbox{ for } 
|\xi|\leq1 \mbox{ and } \eta=0 \mbox{ for } 
|\xi|\geq2$. We use the  characterization of $C^\beta$
in terms of Littlewood-Paley projections (see
 Stein \cite{ste:harmonic}). Indeed,
if a bounded function ${w}$ in $\mathbb{R}^N$ satisfies
\begin{equation*}
\sup_{j=1,2,3,\cdots}
2^{\beta j}
\|\Delta_j({w})\|_{L^\infty(\mathbb{R}^N)}<\infty
\end{equation*}
then ${w}$ lies in $ C^\beta(\mathbb{R}^N)$ and it has the estimate
\begin{equation*}
\|{w}\|_{C^\beta(\mathbb{R}^N)}\leq C_1(\|{w}\|_{L^\infty(\mathbb{R}^N)}+
\sup_{j=1,2,3,\cdots}
2^{\beta j}
\|\Delta_j({w})\|_{L^\infty(\mathbb{R}^N)})
\end{equation*} where $C_1$ depends only on $\beta,N$ and the choice of $\Psi$.
In order to show \eqref{lem_c_beta_(II)}, it is enough to find $0<a<\infty$ such that
$\Psi_{2^{-j}}\in a\mathcal{U}_{2^{-j}}$ for all $j\geq1$ because 
$\Delta_j({w})(x)=\int_{\mathbb{R}^3}{w}(y)\Psi_{2^{-j}}(x-y)dy$
and $\mathcal{U}_r$ is translation invariant.  \\
It is clear that $\Psi$ is a Schwartz function from
the fact $\eta\in C_0^\infty$. Thus we can take  $a:=\|\Psi\|_{L^\infty(\mathbb{R}^3)}+
\|\Psi\|_{L^1(\mathbb{R}^3)}+
\int_{\mathbb{R}^N}|\Psi(x)||x|^\gamma dx<\infty$. Then, 
for any $r>0$, we have
\begin{equation*}\begin{split}
&\int(1/a)\Psi_r =(1/a)\int\Psi =(1/a)\Hat{\Psi}(0)=0,\\
&\|(1/a)\Psi_r\|_{L^\infty} \leq (1/a)r^{-N}
\|\Psi\|_{L^\infty}\leq r^{-N} \leq \frac{A}{r^{N}},\\
&\|(1/a)\Psi_r\|_{L^1} \leq (1/a)
\|\Psi\|_{L^1}\leq 1, \mbox{ and}\\
&\int_{\mathbb{R}^N}|(1/a)\Psi_r(x)||x|^\gamma dx
=(1/a)r^\gamma \cdot\int_{\mathbb{R}^N}|\Psi(x)||x|^\gamma dx\leq r^\gamma.
\end{split} 
\end{equation*}
Thus \eqref{lem_c_beta_(II)} follows with $C=C_1\cdot\max\{1,a\}$.

\end{proof}

We define the backward kernel ${{K}^{(\overline{T})}}$ corresponding to any finite time $\overline{T}<T$ and
to the kernel $K$ by 
\begin{equation}\label{def_back_kernel}{{K}^{(\overline{T})}}(s,x,y)=K(\overline{T}-s,x,y).\end{equation} Then it is easy to see
$T^K_t=T^{{{K}^{(\overline{T})}}}_{\overline{T}-t}$ and they
share the $weak$-$(*)$-kernel condition with the same parameter set.
\begin{lem}\label{lem_dual}
 Let ${w},\varphi
 \in L^\infty(0,\overline{T};(L^1\cap L^\infty)(\mathbb{R}^N))$ be two smooth solutions of \eqref{main_lin_eq}
with $\overline{T}<\infty$ 
for each smooth initial data ${w}_0,\varphi_0\in (L^1\cap L^\infty)(\mathbb{R}^N)$
and
for each associated kernels $K$ and ${{K}^{(\overline{T})}}$,
respectively. In addition, we assume $\varphi_0\in \mathcal{U}_r\cap\mathcal{S}$
for some $r\in(0,1]$.
Then, 
 we have
\begin{equation*}
 \int_{\mathbb{R}^3}{w}_0(x)\varphi(\overline{T},x)dx =
 \int_{\mathbb{R}^3}{w}(\overline{T},x)\varphi_0(x)dx.
\end{equation*}
\end{lem}

\begin{proof}
Let $t\in[0,\overline{T}]$.
Then, we have
\begin{equation*}\begin{split}
&\frac{d}{dt}\int_{\mathbb{R}^3}{w}(t,x)\varphi(\overline{T}-t,x)dx\\&=
\int_{\mathbb{R}^3}(\partial_t{w})(t,x)\varphi(\overline{T}-t,x)dx
-\int_{\mathbb{R}^3}{w}(t,x)(\partial_t\varphi)(\overline{T}-t,x)dx\\
&=-\int_{\mathbb{R}^3}
T^K_t({w}(t,\cdot))(x)
\varphi(\overline{T}-t,x)dx+
\int_{\mathbb{R}^3}
{w}(t,x)
T^{{{K}^{(\overline{T})}}}_{\overline{T}-t}(\varphi(\overline{T}-t,\cdot))(x)dx.
\end{split} 
\end{equation*}
Then, we use Lemma \ref{lem_property_kernel} and the fact
$T^K_t=T^{{{K}^{(\overline{T})}}}_{\overline{T}-t}$ to get 
\begin{equation*}\begin{split}
&=-\int_{\mathbb{R}^3}
T^K_t({w}(t,\cdot))(x)
\varphi(\overline{T}-t,x)dx+
\int_{\mathbb{R}^3}
T^{{{K}^{(\overline{T})}}}_{\overline{T}-t}({w}(t,\cdot))(x)
\varphi(\overline{T}-t,x)dx\\
&=-\int_{\mathbb{R}^3}
T^K_t({w}(t,\cdot))(x)
\varphi(\overline{T}-t,x)dx+
\int_{\mathbb{R}^3}
T^{{K}}_{t}({w}(t,\cdot))(x)
\varphi(\overline{T}-t,x)dx=0.
\end{split} 
\end{equation*}
 As a result, we conclude that $ \int_{\mathbb{R}^3}{w}(t,x)\varphi(\overline{T}-t,x)dx$ is
constant in $t$.
Then put $t:=0$ and $t:=\overline{T}$.

\end{proof}

\section{The main proposition and its proof 
}\label{sec_Proof_of_main_prop}
We are ready to present the main proposition about
the evolution of test functions in a short time interval, whose
length is proportional to $r^\alpha$. 
Roughly speaking, 
if $\varphi_0\in\mathcal{U}_r$, then there exist
$z=z(r,s)$ and $\beta$ such that $\varphi(s)\in\Big(\frac{r}{z}\Big)^\beta\mathcal{U}_z$ for $s\in[0,\delta r^\alpha]$.
\begin{prop}\label{main_prop}
There exist constants $A\geq1$, $\delta>0$, $L>0$ and $\beta>0$ with the following property:\\

Let $0<r\leq1 $ and $\varphi_0\in\mathcal{U}_r\cap\mathcal{S}$.
 Then, there exist a 
 smooth solution  $\varphi
 \in L^\infty(0, {T};(L^1\cap L^\infty)(\mathbb{R}^N))$ of \eqref{main_lin_eq}
with the initial condition
$\varphi(0)=\varphi_0$.
Also,  for 
any $s\in[0,\min\{\delta r^{{\alpha}}, {T}\}]$, we have
\begin{equation}\label{eq_main_prop}
 \varphi(s) \in \Big(\frac{r}{z(r,s)}\Big)^\beta\mathcal{U}_{z(r,s)}
\end{equation} where $z(r,s)$
is defined by $z(r,s)=r(1+L\frac{s}{r^\alpha})$.\\
Moreover, if $r=1$, then \begin{equation}\label{eq_main_prop_r_1}
\varphi(s) \in (1+L{s})\mathcal{U}_{1}\quad \mbox{ for any }  s\in[0, {T}).\end{equation}
\end{prop}
\begin{proof}
Let $\varphi_0\in\mathcal{U}_r\cap\mathcal{S}$ for some $0<r\leq1$. 
Then there exists a weak solution $\varphi$ corresponding to the initial data $\varphi_0$
(this can be proved by following
 \cite{komatsu_fundamental}. Or refer to the approximation scheme
 in \cite{ccv}).
Moreover this solution is smooth,  and it lies in $L_t^\infty(H^b_x)$ for every integer $b\geq0$
due to the smoothness assumption \eqref{smooth_assump_lin} of 
$k$
(it
can be proved by using
 a standard energy argument).\\
 
 First we state  the following elementary inequalities
 without   proof.
 \begin{lem}\label{lem_elemen} 
   (I). $(1-x)\leq\frac{1}{1+x}$ for $x\geq0$.\\
   (II). $(1+\frac{\eta}{2}x)\leq (1+x)^{\eta}$ for 
 any $0\leq x\leq 1$ if $ 0\leq\eta\leq1$.\\
   (III). $(1+x)^{\eta}\leq (1+{\eta}x)$ for any $x\geq0$ if $ 0\leq\eta\leq1$.\\
   (IV). $(1+x)^{\eta}\leq (1+2{\eta}x) $ for any $0<x<C_\eta$  if $ \eta\geq1$.
 \end{lem}

 In order to obtain \eqref{eq_main_prop}, we need to verify 
 the mean zero, the concentration, the $L^\infty$, 
and the $L^1$ conditions. First the mean-zero condition is easily verified
in {STEP 1}.
 Second, we derive  some estimates for remained three other conditions
 in {STEP 2-4}. Then, in {STEP 5}, we combine all the estimates we obtained
 in {STEP 2-4}
 to finish the proof. 
 Without loss of generality, we can assume that a center of $\varphi_0$ is the origin (i.e. $x_0=0$).\\

\textbf{STEP 1.} Mean zero-condition.\\
From (II) of Lemma \ref{lem_property_kernel}, we have, for any $t\in(0,T)$,
\begin{equation*}\begin{split}
 \frac{d}{dt}
\int_{\mathbb{R}^N}  \varphi(t,x)dx
 =& \int_{\mathbb{R}^N} \Big(\frac{\partial}{\partial t}\varphi\Big) (t,x)dx\\
 =& -\int_{\mathbb{R}^N} 
T^K_t(\varphi(t,\cdot))(x)dx=0.\\
\end{split}\end{equation*} 


\textbf{STEP 2.} Concentration-condition.
\begin{lem}\label{lem_concent} 
 There exists a constant
 $C_{conc}>0$
such that,
for any $s\in(0,T)$, we have 
\begin{equation}\label{eq_concent}
  \int_{\mathbb{R}^N} | \varphi(s,x)||x|^\gamma dx\leq 
  r^\gamma(1+
  C_{conc}
  A^{\frac{{{{\alpha}}}-\gamma}{N}}\frac{s}{r^{{{{\alpha}}}}})
\end{equation} where $C_{conc}$ does not 
depend on $A$ as long as $A\geq1$.

\end{lem}
\begin{rem}
This lemma says that test functions lose their concentration
  with certain rate as time goes on.
In Step 5, it will be shown that the  rate
can be absorbed into the regularization effect from the $L^1$ and the $L^\infty$
conditions.

\end{rem}
\begin{proof}

\begin{equation*}\begin{split}
 \frac{d}{ds}
\int_{\mathbb{R}^N} | \varphi(s,x)||x|^\gamma dx
 =& \int_{\mathbb{R}^N} \frac{\partial}{\partial s}\Big(| \varphi(s,x)|\Big)|x|^\gamma dx\\
\leq& \int_{\mathbb{R}^N} -T_s\Big(| \varphi(s,x)|\Big)\cdot|x|^\gamma dx\\
=& \int_{\mathbb{R}^N} -T_s\Big(|x|^\gamma\Big)\cdot | \varphi(s,x)|dx\\
\leq& \int_{\mathbb{R}^N} \Big|T_s\Big(|x|^\gamma\Big)\Big|\cdot
 | \varphi(s,x)|dx=(I)
\end{split}\end{equation*} where we used Lemma \ref{lem_property_equation_with_kernel} and Lemma \ref{lem_property_kernel}.\\
First, consider  the case $\alpha<1$. Then, thanks to Lemma \ref{lem_fractional_laplacian}, we have
\begin{equation*}\begin{split}
(I)\leq &C
\int_{\mathbb{R}^N} |x|^{{\gamma}-{{{\alpha}}}}\cdot
(1+|x|^\omega)\cdot | \varphi(s,x)|dx\\
=& C
\Big(\int_{B(A^{-1/N}r)} |x|^{{\gamma}-{{{\alpha}}}}\cdot 
(1+|x|^\omega)\cdot| \varphi(s,x)|dx
\\
&\quad\quad\quad +\int_{B(A^{-1/N}r)^C}|x|^{{\gamma}-{{{\alpha}}}}\cdot
  | \varphi(s,x)|dx\\
 &\quad\quad\quad +\int_{B(A^{-1/N}r)^C}|x|^{{\gamma}-{{{\alpha}}}+\omega}\cdot
   | \varphi(s,x)|dx\Big).
   \end{split}\end{equation*} 
From the condition $\gamma<(\alpha-\omega)$,
we have
decreasing of the functions
 $|\cdot|^{{\gamma}-{{{\alpha}}}}$ 
and  $|\cdot|^{{\gamma}-{{{\alpha}}}+\omega}$.
Also, note that 
 $L^\infty$ and $L^1$ norms are decreasing
and $A^{-{1}/{N}}\cdot r\leq 1$ from $A\geq 1$ and $r\leq 1$. Thus we have
   \begin{equation*}\begin{split}
\leq& C
\Big(
(A^{-1/N}r)^{\gamma-{{\alpha}}+N} \cdot
(1+(A^{-1/N}r)^\omega)\cdot \| \varphi(s)\|_{L^\infty}\\
&\quad\quad\quad +(A^{-1/N}r)^{\gamma-{{\alpha}}}\| \varphi(s)\|_{L^1}
+(A^{-1/N}r)^{\gamma-{{\alpha}}+\omega}\| \varphi(s)\|_{L^1}\Big)\\
\leq& C
 A^{\frac{{{\alpha}}-\gamma}{N}}r^{\gamma-{{\alpha}}}.\\
\end{split}\end{equation*} 
Likewise, for the case $\alpha\geq1$, Lemma 
\ref{lem_fractional_laplacian}
with $\gamma<\alpha-(\nu+\omega)$
 gives us the same conclusion. 
Then, we have
\eqref{eq_concent} thanks to the initial condition $\int_{\mathbb{R}^N} | \varphi(0,x)||x|^\gamma dx\leq r^\gamma$.
\end{proof}

\textbf{STEP 3.} $L^\infty$-condition.
\begin{lem}\label{lem_L_infty}
 There exist two constants $\delta_{L^\infty}>0$
and $C_{L^\infty}>0$
such that,
for any $s\in[0,\min\{\delta_{L^\infty} r^{{\alpha}},T\}]$, we have
\begin{equation}\label{eq_l_infty}
  \|\varphi(s,\cdot)\|_{L^\infty({\mathbb{R}^N})}\leq \frac{A}{r^N}(1-
  C_{L^\infty}A^{\frac{{\alpha}}{N}}
\frac{s}{r^{{\alpha}}})
\end{equation}  where $C_{L^\infty}$ does not 
depend on $A$ as long as $A\geq1$.

\end{lem}
\begin{rem}
This lemma is proved by using the lower bound of \eqref{cond_bounds},
which gives us some regularization effect.
 We follow a similar argument of Theorem 4.1 in
 the paper 
 C{\'o}rdoba  and C{\'o}rdoba \cite{cordoba_maximum}, 
  which
showed a $L^\infty$ decay for smooth solutions of the 2D surface QG equation. \\

\end{rem}
\begin{proof}

First, we define
$M(t):=\|\varphi(t,\cdot)\|_{L^\infty}$. 
We claim that
there exist $\delta_{1}>0$ and $C_1>0$ such that
for any $t\in[0,\delta_{1}r^\alpha]$ satisfying
$M(t)\geq \frac{1}{2}\frac{A}{r^N}$, we have
 \begin{equation}\label{claim_infty}
M(t)\leq
\frac{A}{r^N}(1-
 C_1A^{\frac{{\alpha}}{N}}
\frac{t}{r^{{\alpha}}}).\end{equation}


To prove the above claim \eqref{claim_infty}, first pick any $t\in(0,T)$ such that 
\begin{equation*}M(t)\geq \frac{1}{2}\cdot\frac{A}{r^N}.\end{equation*}
Then we know $M(t^*)\geq \frac{1}{2}\frac{A}{r^N}$ for all $t^*<t$
from Lemma \ref{lem_property_equation_with_kernel}.
It can be easily proved that there exists a point $x_t$ such that $|\varphi(t,x_t)|=M(t)$. Indeed,
because our kernel lies in $C^\infty_{t,x,y}\Big([0,T]\times\overline{\mathbb{R}^N}\times\overline{\mathbb{R}^N})\Big)$
with $\varphi_0\in\mathcal{S}$, we can show $\varphi\in L^\infty_t H^d$ for every integer $d\geq 0$
by standard energy estimates. In particular, $\varphi(t,\cdot)\in H^b$ for some integer $b>(N/2)$ for every time. Then,
$\varphi(t,\cdot)$ vanishes at the infinity thanks to a Fourier transform argument. Since $\varphi(t,\cdot)$
is continuous, there exists a maximum (or minimum) point.\\

Then, for almost every time $t\in(0,T)$, there exist a point $\tilde{x}_{t}$ 
such that
 $|\varphi(t,\tilde{x}_{t})|=M(t)$ 
 with
  the following inequality:
\begin{equation*}
\frac{d}{dt}M(t)\leq 
\begin{cases}
&\Big(\frac{\partial \varphi}{\partial t}\Big)(t,\tilde{x}_{t}) 
     \quad\mbox{ if } \varphi(t,\tilde{x}_{t})=M(t),\\ 
&-\Big(\frac{\partial \varphi}{\partial t}\Big)(t,\tilde{x}_{t}) 
     \quad\mbox{ if } \varphi(t,\tilde{x}_{t})=-M(t)                
\end{cases}
\end{equation*} (this can be proved
by following
the argument of \cite{cordoba_maximum}).\\

We assume the first case $\varphi(t,\tilde{x}_{t})=M(t)>0$ (the other one
can be dealt in similar fashion). 
Then
\begin{equation}\begin{split}\label{infty_inequal}
\frac{d}{dt}M(t)&\leq 
\Big(\frac{\partial \varphi}{\partial t}\Big)(t,\tilde{x}_{t}) 
\leq -T_t^K(\varphi(t,\cdot))(x)\\
&=-\int_{\mathbb{R}^N}\Big({\varphi(t,\tilde{x}_{t})-\varphi(t,y)}\Big)K(t,\tilde{x}_{t},y)dy\\
&\leq-{\Lambda^{-1}}\int_{|\tilde{x}_{t}-y|\leq\zeta}\frac{\varphi(t,\tilde{x}_{t})
-\varphi(t,y)}{|\tilde{x}_{t}-y|^{N+\alpha}}dy=-{\Lambda^{-1}}\cdot(*).\\
  \end{split} 
   \end{equation} 
We used the fact ${\varphi(t,\tilde{x}_{t})-\varphi(t,y)}\geq 0$
with the lower bound of the kernel \eqref{cond_bounds}.\\

Let $R$ be any number between $0$ and $\zeta$, which will be
chosen soon. We separate the ball $B_R(\tilde{x}_{t})$
into two disjoint regions $\Omega_1$ and $\Omega_2$ by the following way:
$\Big({\varphi(t,\tilde{x}_{t})-\varphi(t,y)}\Big)>\frac{1}{2}
\varphi(t,\tilde{x}_{t})$  implies $y\in\Omega_1$. Otherwise, $y\in\Omega_2$. Then we have 
the following upper bound
of measure of $\Omega_2$: 
\begin{equation*}
|\Omega_2|
=\frac{2}{M(t)}\int_{\Omega_2}\frac{M(t)}{2}dy
\leq\frac{2}{M(t)}\int_{\Omega_2}\varphi(t,y)dy
\leq\frac{2}{M(t)}\|\varphi(t,\cdot)\|_{L^1(\mathbb{R}^N)}\leq\frac{2}{M(t)}.
\end{equation*} As a result,
from $R<\zeta$, we have
\begin{equation*}\begin{split}
(*)&\geq
\int_{|\tilde{x}_{t}-y|\leq R}\frac{\varphi(t,\tilde{x}_{t})
-\varphi(t,y)}{|\tilde{x}_{t}-y|^{N+\alpha}}dy 
\geq
\int_{\Omega_1}\frac{\varphi(t,\tilde{x}_{t})
-\varphi(t,y)}{|\tilde{x}_{t}-y|^{N+\alpha}}dy  \\
&\geq \frac{M(t)}{2R^{N+\alpha}}|\Omega_1|
= \frac{M(t)}{2R^{N+\alpha}}(|B_R(\tilde{x}_t)|-|\Omega_2|) 
\geq \frac{M(t)}{2R^{N+\alpha}}(V_N R^N-\frac{2}{M(t)}).
  \end{split} 
   \end{equation*}

Now we choose $R$ by  $V_N R^N=\frac{2}{M(t)}\cdot 2$. Then, it is clear
that $R\leq\zeta$ 
because $M(t)\geq \frac{1}{2}\frac{A}{r^N}\geq
\frac{1}{2}\frac{1}{r^N}\geq \frac{1}{2}$ and
$ R=\Big(\frac{4}{M(t)V_N}\Big)^{1/N}\leq
 \Big(\frac{8}{V_N}\Big)^{1/N}\leq\zeta$ by \eqref{def_zeta_0}.
Coming
back to \eqref{infty_inequal}, we have
\begin{equation*}\begin{split}
\frac{d}{dt}M(t)
&\leq -{\Lambda^{-1}}\frac{M(t)}{2R^{N+\alpha}}(V_N R^N-\frac{2}{M(t)})
= -{\Lambda^{-1}}\frac{M(t)}{2R^{N+\alpha}}\cdot\frac{2}{M(t)}
=-{\Lambda^{-1}}\frac{1}{R^{N+\alpha}}\\
&=-{\Lambda^{-1}}  {\Big(\frac{M(t)V_N}{4}\Big)^{(N+\alpha)/N}}
=-
C
\cdot M(t)^{1+\frac{\alpha}{N}}.
  \end{split} 
   \end{equation*} Solving this differential inequality, we obtain
\begin{equation*}\begin{split}
M(t)
&\leq {M(0)}{\Big(1+
C
\cdot M(0)^{\alpha/N}\cdot t\Big)^{-N/\alpha}}\\
\end{split} 
   \end{equation*} From the fact $\frac{1}{2}\frac{A}{r^N}\leq M(t)\leq M(0)\leq \frac{A}{r^N}$, we have
   \begin{equation*}\begin{split}
&\leq \frac{A}{r^N}{\Big(1+
C
\cdot {\Big(\frac{1}{2}\cdot\frac{A}{r^N}\Big)}^{\alpha/N}\cdot t\Big)^{-N/\alpha}}\\
\end{split} 
   \end{equation*}
  For any $p>0$, it is easy to see $(1+x)^{-p}\leq(1-\frac{1}{2}p x)$ for $0\leq x\leq C_p$. Thus, we have
   \begin{equation*}\begin{split}
&\leq \frac{A}{r^N}{\Big(1-
C_1
\cdot A^{\alpha/N}\cdot \frac{t}{r^\alpha}\Big)}\\  
\end{split} 
   \end{equation*}
 as long as $  {t}\leq C_1^{-1}A^{-\alpha/N}{r^\alpha}$.   
By taking  $\delta_1:= C_1^{-1}A^{-\alpha/N}$,
we proved the claim \eqref{claim_infty}
under the assumption $M(t)\geq \frac{1}{2}\frac{A}{r^N}$.\\

Thanks to \eqref{claim_infty}, the whole case \eqref{eq_l_infty} 
 can be achieved easily by taking 
 $\delta_{L^\infty}:=
 \frac{1}{2}\delta_1$ and 
 $C_{L^\infty}:=C_1$.
  Indeed, 
if $M(t)\leq \frac{1}{2}\frac{A}{r^N}$, then
we have
\begin{equation*}
M(t)
\leq\frac{1}{2}\frac{A}{r^N}
\leq\frac{A}{r^N}(1-
C_1A^{\frac{{\alpha}}{N}}
\frac{t}{r^{{\alpha}}})
\end{equation*} as long as 
 $  {t}\leq \frac{1}{2}C_1^{-1}A^{-\alpha/N}{r^\alpha}=\frac{1}{2}\delta_1{r^\alpha}$.\\
 


\end{proof}

\textbf{STEP 4.} $L^1$-condition.
\begin{lem}\label{lem_L_1}
 There exist two constants $\delta_{L^1}>0$
and $C_{L^1}>0$
such that,
for any $s\in[0,\min\{\delta_{L^1} r^{{\alpha}},T\}]$, we have
\begin{equation}\label{eq_L_1}
  \|\varphi(s,\cdot)\|_{L^1({\mathbb{R}^N})}\leq (1-
 C_{L^1}\cdot\frac{s}{r^{{\alpha}}})
\end{equation}  where $C_{L^1}$ does not 
depend on $A$ as long as $A\geq1$.
\end{lem}

\begin{rem}
In this time, we obtain $L^1$ decay by using the lower bound
of the kernel \eqref{cond_bounds}.
In general, without the mean zero property,
we do not expect $L^1$ decay (refer to \cite{cordoba_maximum}).
However, with the mean zero property, we can manage
certain  amount of cancellation
of the $L^1$-norm.
 This idea comes from the argument in \cite{variation_kiselev} where
 $L^1$ decay for
mean-zero solutions for the 2D-SQG equation in a periodic setting was obtained.
\end{rem}

\begin{proof}
First, by using \eqref{eq_concent}, we can find
$\delta_{2}$ such that 
$ \int_{\mathbb{R}^N} | \varphi(s,x)||x|^\gamma dx\leq \frac{11}{10}r^\gamma$
 for all $t\in[0,\delta_{2}r^\alpha]$.
i.e. we take $\delta_2$ so small that
$C_{conc}
A^{\frac{{{{\alpha}}}-\gamma}{N}}\delta_{2}
\leq\frac{1}{10}$.\\ 

We claim that
there exists a constant $C_2>0$ such that
for any $t\in[0,\delta_{2}r^\alpha]$ satisfying
$\|\varphi(t)\|_{L^1(\mathbb{R}^N)}\geq 
\frac{9}{10}$, we have
 \begin{equation}\label{claim_L_1}
 \|\varphi(t,\cdot)\|_{L^1({\mathbb{R}^N})}\leq (1-
  C_2\cdot\frac{t}{r^{{\alpha}}}).\end{equation}

To prove \eqref{claim_L_1}, let $t\in[0,\delta_{2}r^\alpha]$ satisfy 
$\|\varphi(t,\cdot)\|_{L^1(\mathbb{R}^N)}\geq \frac{9}{10}$.
For simplicity, we define $a:=(11)^{1/\gamma}$. 
Then, from \eqref{def_zeta_0}, we know  \begin{equation}\label{2a_leq_zeta}
              2a\leq\zeta  
             \end{equation}  
 and  the following estimates hold:
\begin{equation}\label{L_1_estimate_ball_ar}
 \int_{B(ar)}|\varphi(t,x)|dx\geq\frac{8}{10},
\int_{B(ar)}\varphi^+(t,x)dx\geq\frac{3}{10}, \mbox{ and }
\int_{B(ar)}\varphi^-(t,x)dx\geq\frac{3}{10}.
\end{equation} where $f^+:=\max\{f,0\}$ and $f^-:=\max\{-f,0\}$.

Indeed, from  the concentration condition, 
we obtain the following upper bound of $L^1$-norm outside of the ball $B(ar)$: 
\begin{equation*}
 \int_{B(ar)^C}|\varphi(t,x)|dx
=\int_{B(ar)^C}|\varphi(t,x)||x|^\gamma |x|^{-\gamma} dx\leq 
\frac{11}{10}r^\gamma (ar)^{-\gamma}=\frac{1}{10}.
\end{equation*} 
Then, thanks to the mean-zero property, 
we get the following lower bounds of $\|\varphi\|_{L^1(B(ar))}$,
$\|\varphi^+\|_{L^1(B(ar))}$ and
$\|\varphi^-\|_{L^1(B(ar))}$:
\begin{equation*}
 \int_{B(ar)}|\varphi(t,x)|dx
=\|\varphi(t,\cdot)\|_{L^1(\mathbb{R}^N)}-
\int_{B(ar)^C}|\varphi(t,x)| dx\geq 
\frac{9}{10}-\frac{1}{10}=\frac{8}{10},
\end{equation*} 
\begin{equation*}\begin{split}
& \int_{B(ar)}\varphi^{\pm}(t,x)dx
=\int_{\mathbb{R}^N}\varphi^{\mp}(t,x)dx-
\int_{B(ar)^C}\varphi(t,x)^{\pm} dx \\ 
&= 
\frac{1}{2}\|\varphi(t,\cdot)\|_{L^1(\mathbb{R}^N)}-
\int_{B(ar)^C}\varphi(t,x)^{\pm} dx\\
&\geq 
\frac{1}{2}\|\varphi(t,\cdot)\|_{L^1(\mathbb{R}^N)}-
\int_{B(ar)^C}|\varphi(t,x)| dx
\geq 
\frac{9}{20}-\frac{1}{10}>\frac{3}{10}.
\end{split}\end{equation*} 
We denote symbols $D^s_+,D^s_-$ and $S^s$ by
\begin{equation*}\begin{split}                 
 &D^s_{\pm}=\{x\in\mathbb{R}^N |\quad  \pm \varphi(s,x)\geq 0\} 
 \mbox{ and}\\
&S^s=\{x\in\mathbb{R}^N |\quad   \varphi(s,x)= 0\}.
                \end{split}
\end{equation*}
Then, we have
\begin{equation*}\begin{split}                 
&\frac{d}{dt}\int_{\mathbb{R}^N} |\varphi(s,x)|dx 
=\int_{\mathbb{R}^N} \frac{\partial}{\partial t}|\varphi(s,x)|dx
=\int_{\mathbb{R}^N} \Big(1_{D^s_+}(x)-1_{D^s_-}(x)\Big)(\partial_t
\varphi)(s,x)dx\\
&=-\int_{\mathbb{R}^N} \Big(1_{D^s_+}(x)-1_{D^s_-}(x)\Big)
T^K_s(\varphi(s,\cdot))(x)
dx\\
&=-\int_{\mathbb{R}^N} \Big(1_{D^s_+}(x)-1_{D^s_-}(x)\Big)
\int_{\mathbb{R}^N}[\varphi(s,x)-\varphi(s,y)]K(s,x,y)dy
dx\\
&=-\frac{1}{2}\int_{\mathbb{R}^N} \int_{\mathbb{R}^N}
\Big[\Big(1_{D^s_+}(x)-1_{D^s_-}(x)\Big)
-\Big(1_{D^s_+}(y)-1_{D^s_-}(y)\Big)\Big]
\Big({\varphi(s,x)-\varphi(s,y)}\Big)\cdot\\
&\quad\quad\quad\quad\quad\quad\quad\quad\quad\quad\quad\quad
\quad\quad\quad\quad\quad\quad\quad\quad\quad\quad
K(s,x,y)dy
dx\\
&=-\frac{1}{2}\int_{\mathbb{R}^N} \int_{\mathbb{R}^N}
(*)dy
dx\\
 \end{split}\end{equation*} where we use the simplification $(*):=$
\begin{equation*}
\Big[\Big(1_{D^s_+}(x)-1_{D^s_-}(x)\Big)
-\Big(1_{D^s_+}(y)-1_{D^s_-}(y)\Big)\Big]
\Big({\varphi(s,x)-\varphi(s,y)}\Big)\cdot K(s,x,y).
\end{equation*} Then, we split the above integral into 9 components:
\begin{equation*}\begin{split}
&=-\frac{1}{2}\Big[
\int_{D^s_+} \int_{D^s_+}(*)dydx
+\int_{D^s_+} \int_{S^s}(*)dydx
+\int_{D^s_+} \int_{D^s_-}(*)dydx\\
&\quad\quad\quad+\int_{S^s} \int_{D^s_+}(*)dydx
+\int_{S^s} \int_{S^s}(*)dydx
+\int_{S^s} \int_{D^s_-}(*)dydx\\
&\quad\quad\quad+\int_{D^s_-} \int_{D^s_+}(*)dydx
+\int_{D^s_-} \int_{S^s}(*)dydx
+\int_{D^s_-} \int_{D^s_-}(*)dydx
\Big]\\
&=-\frac{1}{2}\Big[
(I)+(II)+\cdots+(IX)
\Big]\\
 \end{split}\end{equation*} 

We will prove the inequality: 
$\frac{1}{2}\Big[
(I)+(II)+\cdots+(IX)
\Big]\geq 
 C_2\cdot r^{-\alpha}$, which will imply the claim \eqref{claim_L_1} later. First,
we observe that $(I)=(V)=(IX)=0$ by the definition of $(*)$.\\

Second, we have $(II)=(IV)$ by symmetry of the kernel. Indeed,
\begin{equation*}\begin{split}  
&(II)=\int_{D^s_+} \int_{S^s}(*)dydx=
\int_{D^s_+} \int_{S^s}\varphi(s,x)\cdot K(s,x,y)
dydx\\
&= \int_{S^s}\int_{D^s_+}\varphi(s,x)\cdot K(s,x,y)
dxdy
= \int_{S^s}\int_{D^s_+}\varphi(s,y)\cdot K(s,y,x)
dydx\\
&= \int_{S^s}\int_{D^s_+}\varphi(s,y)\cdot K(s,x,y)
dydx=(IV).\\
 \end{split}\end{equation*}
Likewise, we have $(VI)=(VIII)$ and $(III)=(VII)$.
Thus, we have 
\begin{equation*}\begin{split} 
&\Big[
(I)+(II)+\cdots+(IX)
\Big]=2\Big[(IV)+(VI)+(III)\Big]\\
&=2\Big[
\int_{S^s} \int_{D^s_+}(*)dydx
+\int_{S^s} \int_{D^s_-}(*)dydx
+\int_{D^s_+} \int_{D^s_-}(*)dydx
\Big]\\
&=2\Big[
\int_{S^s} \int_{D^s_+\cup D^s_-}(*)dydx
+\int_{D^s_+} \int_{D^s_-}(*)dydx
\Big]\\
&=2\Big[
\int_{S^s} \int_{\mathbb{R}^N}(*)dydx
+\int_{D^s_+} \int_{D^s_-}(*)dydx
\Big]=2[(D)+(B)]
 \end{split}\end{equation*} where the third equality follows
$\int_{S^s} \int_{S^s}(*)dydx=(V)=0$.\\

In order to use the lower bound of the kernel \eqref{cond_bounds},
we need to restrict the above integral on a subset of $\{|x-y|\leq\zeta\}$. 
For this purpose, we define the subsets
 $\tilde{D}^s_+,\tilde{D}^s_-$ and $\tilde{S}^s$ by
$\tilde{D}^s_{\pm}= D^s_{\pm}\cap B(ar)$ and $\tilde{S}^s=S_s\cap B(ar)$.
Then, if $x,y\in B(ar)$, then $|x-y|\leq 2ar\leq 2a\leq\zeta$
from \eqref{2a_leq_zeta}.
 Thus, from 
 the lower bound of the kernel \eqref{cond_bounds}, 
 we have
\begin{equation*}\begin{split} 
&(D)=
\int_{S^s} \int_{\mathbb{R}^N}(*)dydx\\
&=\int_{S^s} \int_{\mathbb{R}^N}
\Big[(0-0)
-\Big(1_{D^s_+}(y)-1_{D^s_-}(y)\Big)\Big]
\Big({-\varphi(s,y)}\Big)\cdot K(s,x,y)dydx\\
&=\int_{S^s} \int_{\mathbb{R}^N}|{\varphi(s,y)}|\cdot K(s,x,y)dydx\\
&\geq
{\Lambda^{-1}}
\int_{\tilde{S}^s} \int_{B(ar)}\frac{|{\varphi(s,y)}|}{|x-y|^{N+\alpha}}
 dydx
\geq{\Lambda^{-1}}(2ar)^{-(N+\alpha)}
\int_{\tilde{S}^s} \|{{\varphi(s,\cdot)}}\|_{L^1({B(ar)})}
dx
\\&={\Lambda^{-1}}\cdot(2ar)^{-(N+\alpha)}\cdot
|{\tilde{S}^s}|\cdot \|{{\varphi(s,\cdot)}}\|_{L^1({B(ar)})}
\geq{\Lambda^{-1}}\cdot(2ar)^{-(N+\alpha)}\cdot
|{\tilde{S}^s}|\cdot\frac{8}{10}
 \end{split}\end{equation*} where, 
for the last equality,  the estimate 
\eqref{L_1_estimate_ball_ar} was used.\\

Also, we have
\begin{equation*}\begin{split} 
&(B)=\int_{D^s_+} \int_{D^s_-}(*)dydx\\
&=\int_{D^s_+} \int_{D^s_-}\Big[\Big(1_{D^s_+}(x)-0\Big)
-\Big(0-1_{D^s_-}(y)\Big)\Big]
\Big({\varphi(s,x)-\varphi(s,y)}\Big)\cdot K(s,x,y)dydx\\
&=2\int_{D^s_+} \int_{D^s_-}
\Big({\varphi(s,x)-\varphi(s,y)}\Big)\cdot K(s,x,y)dydx\\
&=2\Big[\int_{D^s_+} \int_{D^s_-}
\varphi(s,x)\cdot K(s,x,y)dydx
+\int_{D^s_+} \int_{D^s_-}
-\varphi(s,y)\cdot K(s,x,y)dydx\Big]\\
&=2[(B_1)+(B_2)].
 \end{split}\end{equation*}

We can obtain the following lower bound of $(B_1)$ by the following way:
\begin{equation*}\begin{split} 
&(B_1)=\int_{D^s_+} \int_{D^s_-}
\varphi(s,x)\cdot K(s,x,y)dydx=
\int_{D^s_-} \int_{D^s_+}
\varphi(s,x)\cdot K(s,x,y)dxdy\\
& 
\geq {\Lambda^{-1}}
\int_{\tilde{D}^s_-} \int_{\tilde{D}^s_+}
\frac{\varphi(s,x)}{|x-y|^{N+\alpha}}\cdot
dxdy
\geq {\Lambda^{-1}} (2ar)^{-(N+\alpha)}
\int_{\tilde{D}^s_-} \|{{\varphi^+(s,\cdot)}}\|_{L^1({B(ar)})}dy\\
&\geq 
{\Lambda^{-1}} (2ar)^{-(N+\alpha)}\cdot
|{\tilde{D}^s_-}| \cdot \frac{3}{10}.
 \end{split}\end{equation*} Likewise, for $(B_2)$,
 we have $(B_2)\geq
{\Lambda^{-1}} \cdot(2ar)^{-(N+\alpha)}\cdot
|{\tilde{D}^s_+}| \cdot \frac{3}{10}$.\\

Now we have a desirable estimate for $\Big[
(I)+(II)+\cdots+(IX)
\Big]$:
\begin{equation*}\begin{split} 
&\frac{1}{2}\Big[
(I)+(II)+\cdots+(IX)
\Big]=\Big[(IV)+(VI)+(III)\Big]=[(D)+(B)]\\
&\geq \Big[
{\Lambda^{-1}}\cdot(2ar)^{-(N+\alpha)}\cdot
|{\tilde{S}^s}|\cdot\frac{8}{10}
\\
&\quad\quad\quad\quad\quad+2\Big(
{\Lambda^{-1}}\cdot (2ar)^{-(N+\alpha)}\cdot
|{\tilde{D}^s_-}| \cdot \frac{3}{10}
+{\Lambda^{-1}}\cdot (2ar)^{-(N+\alpha)}\cdot
|{\tilde{D}^s_+}| \cdot \frac{3}{10}
\Big)
\Big]\\
&\geq{\Lambda^{-1}}\cdot (2ar)^{-(N+\alpha)}\cdot\frac{6}{10}\cdot
\Big(
|{\tilde{S}^s}|+|{\tilde{D}^s_-}|+|{\tilde{D}^s_+}|
\Big)\\
&\geq{\Lambda^{-1}}\cdot (2ar)^{-(N+\alpha)}\cdot\frac{6}{10}\cdot
\Big(
(ar)^{N}\cdot V_N
\Big)=
 C_2\cdot r^{-\alpha}.
\\
 \end{split}\end{equation*}

It prove the claim \eqref{claim_L_1}, under the assumption
$\|\varphi(t)\|_{L^1(\mathbb{R}^N)}\geq 
\frac{9}{10}$,
 because
\begin{equation*}\begin{split} 
\|\varphi(t)\|_{L^1(\mathbb{R}^N)}&=
\|\varphi(0)\|_{L^1(\mathbb{R}^N)}+
\int_0^t\frac{d}{ds}\|\varphi(s,\cdot)\|_{L^1(\mathbb{R}^N)}ds\\
&\leq 1+ 
 C_2\cdot r^{-\alpha}\cdot t 
\quad\mbox{ for any } t\in[0,\delta_2r^\alpha].
\end{split}\end{equation*}

On the other hand, if $\|\varphi(t)\|_{L^1(\mathbb{R}^N)}\leq 
\frac{9}{10}$, then 
we have
\begin{equation*}
\|\varphi(t)\|_{L^1(\mathbb{R}^N)}
\leq\frac{9}{10}\leq
(1-
 C_2\cdot\frac{t}{r^{{\alpha}}})
\end{equation*} as long as
 $ 
  t\leq\frac{1}{10}C_2^{-1}r^{\alpha}$.
Therefore, by taking $\delta_{L^1}:=\min\{\delta_2,\frac{1}{10}C_2^{-1}\}$ and $C_{L^1}:=C_2$,
we finish the proof of Lemma \ref{lem_L_1}.
\\ 

\end{proof}

\textbf{STEP 5.} Combining all conditions.\\

Now we are ready to finish the proof of the main 
proposition \ref{main_prop}.
In {STEP 2-4}, we proved that
\begin{equation}\label{step5_concent}
  \int_{\mathbb{R}^N} | \varphi(s,x)||x|^\gamma dx\leq 
  r^\gamma(1+
  C_{conc} 
   A^{\frac{{{{\alpha}}}-\gamma}{N}}\frac{s}{r^{{{{\alpha}}}}})
\quad\mbox{ for } s\in[0,T),
\end{equation} 
\begin{equation}\label{step5_L_infty}
  \|\varphi(s,\cdot)\|_{L^\infty({\mathbb{R}^N})}
\leq \frac{A}{r^N}(1-
 C_{L^\infty}A^{\frac{{\alpha}}{N}}
\frac{s}{r^{{\alpha}}}) 
\quad\mbox{ for } s\in[0,\min\{\delta_{L^\infty}\cdot r^\alpha,T\}], \mbox{ and}
\end{equation}
\begin{equation}\label{step5_L_1}
  \|\varphi(s,\cdot)\|_{L^1({\mathbb{R}^N})}\leq (1-
   C_{L^1}\cdot\frac{s}{r^{{\alpha}}}) 
\quad\mbox{ for } s\in[0,\min\{\delta_{L^1}\cdot r^\alpha,T\}].
\end{equation}

Note that the constants $C_{L^1},C_{L^\infty},$ and $C_{conc}$
are independent of $A$ as long as $A\geq1$ while 
$\delta_{L^1}$ and $\delta_{L^\infty}$ depend on $A$.
We define $\delta_3:=\min\{\delta_{L^1},\delta_{L^\infty}\}$
so that the above three estimates \eqref{step5_concent},
\eqref{step5_L_infty}, and \eqref{step5_L_1} hold at the same time for all 
$s\in[0,\min\{\delta_{3}r^\alpha,T\}]$.
Without loss of generality, we can assume
$C_{L^1}=C_{L^\infty}\leq C_{conc}$.\\

Recall that we are looking for $\beta>0$ and  $z(r,s)$ such that
$\varphi(s)\in\Big(\frac{r}{z}\Big)^\beta\mathcal{U}_z$.
Thus, from Definition \ref{def_u_r} of $\mathcal{U}_r$ 
and from the above three estimates \eqref{step5_concent},
\eqref{step5_L_infty}, and \eqref{step5_L_1}, 
we need the followings:
\begin{equation}\begin{split}\label{goal_z_beta}
r^\gamma(1+
  C_{conc} 
   A^{\frac{{{{\alpha}}}-\gamma}{N}}\frac{s}{r^{{{{\alpha}}}}})
\leq& \Big(\frac{r}{z}\Big)^\beta z^\gamma,\\
\frac{A}{r^N}(1-
 C_{L^\infty}A^{\frac{{\alpha}}{N}}
\frac{s}{r^{{\alpha}}}) \leq& \Big(\frac{r}{z}\Big)^\beta\cdot
\frac{A}{z^N}, \mbox{ and}\\
(1-
   C_{L^1}\cdot\frac{s}{r^{{\alpha}}}) \leq& \Big(\frac{r}{z}\Big)^\beta.
\end{split}\end{equation} 
\begin{rem}\label{rem_competition}
\eqref{goal_z_beta} is equivalent to
\begin{equation}\begin{split}\label{concent_rough}
&r(1+
  C_{conc} 
   A^{\frac{{{{\alpha}}}-\gamma}{N}}\frac{s}{r^{{{{\alpha}}}}})^{1/(\gamma-\beta)}
\leq z,
\end{split}\end{equation} 
\begin{equation}\begin{split}\label{concent_L_infty}
&z \leq r(1-
 C_{L^\infty}A^{\frac{{\alpha}}{N}}
\frac{s}{r^{{\alpha}}})^{-1/(N+\beta)}, \mbox{ and}
\end{split}\end{equation} 
\begin{equation}\begin{split}\label{concent_L_1}
&z \leq r(1-
   C_{L^1}\cdot\frac{s}{r^{{\alpha}}})^{-1/\beta}.
\end{split}\end{equation} 
The power of $A$ in \eqref{concent_rough} 
is strictly less than that of $A$ in  \eqref{concent_L_infty}.
This fact is crucial because we can
choose $A$ large enough to hold \eqref{concent_rough} 
and \eqref{concent_L_infty} at the same time. Then we can
make $\beta$ small enough to hold \eqref{concent_L_1}, too.
We will now give all the details.\\
\end{rem}

We  take any $A\geq1$ large enough to satisfy the inequality:
\begin{equation*}
\frac{8}{\gamma}({N+(1/2)})\cdot
C_{conc}
A^{\frac{{{{\alpha}}}-\gamma}{N}} 
\leq 
 C_{L^\infty}A^{\frac{{\alpha}}{N}}.
\end{equation*}
 In addition, we take any $\beta\in(0,\gamma/2]$ so small that
 the following inequality holds:
\begin{equation*}
 \frac{4}{\gamma}\beta\cdot
C_{conc} 
 A^{\frac{{{{\alpha}}}-\gamma}{N}}
 \leq 
  C_{L^1}.
\end{equation*}
Finally, we define a constant $L$ by 
\begin{equation*}
L:=\frac{2}{\gamma-\beta}\cdot
C_{conc}
A^{\frac{{{{\alpha}}}-\gamma}{N}}
\end{equation*}
and a function $z(r,s)$ by 
\begin{equation*}
z(r,s):={r(1+L\frac{s}{r^\alpha})}.
\end{equation*}

For the Concentration-condition, 
from
 (II) of Lemma \ref{lem_elemen}, we have
\begin{equation*}\begin{split}
 &\int_{\mathbb{R}^N} | \varphi(s,x)||x|^\gamma dx\leq  
 r^\gamma(1+
C_{conc}
 A^{\frac{{{{\alpha}}}-\gamma}{N}}\frac{s}{r^{{{{\alpha}}}}})
= r^\gamma
\Big[1+\frac{\gamma-\beta}{2}\cdot L\cdot\frac{s}{r^\alpha}\Big]\\
&\leq r^\gamma
\Big[1+L\cdot\frac{s}{r^\alpha}\Big]^{\gamma-\beta}
= \Big(\frac{r}{z}\Big)^\beta \cdot
z^\gamma
\end{split}\end{equation*} where the last inequality
holds as long as 
for $s\leq (1/L) r^\alpha $. 
We define $\delta_4:=\min\{\delta_3,(1/L)
\}$.\\

On the other hand, from $0<\beta\leq\gamma/2<1/2$, we observe the followings:
\begin{equation}\begin{split}\label{ineq_L_infty}
2\cdot({N+\beta})\cdot L
&\leq\frac{8}{\gamma}({N+(1/2)})\cdot
C_{conc}
 A^{\frac{{{{\alpha}}}-\gamma}{N}} 
\leq
 C_{L^\infty}A^{\frac{{\alpha}}{N}}\quad\mbox{and}
\end{split}\end{equation}
\begin{equation}\begin{split}\label{ineq_L_1}
 \beta\cdot L
&\leq\frac{4}{\gamma} \beta\cdot 
C_{conc}
 A^{\frac{{{{\alpha}}}-\gamma}{N}}
 \leq 
  C_{L^1}.
\end{split}\end{equation}

For the $L^\infty$-condition,
from \eqref{ineq_L_infty} and from
(I) and (IV) of Lemma \ref{lem_elemen}, we have
\begin{equation*}\begin{split}
 & \|\varphi(s,\cdot)\|_{L^\infty({\mathbb{R}^N})}
\leq \frac{A}{r^N}(1-
 C_{L^\infty}A^{\frac{{\alpha}}{N}}
\frac{s}{r^{{\alpha}}})
\leq \frac{A}{r^N}
\cdot  \Big[{1-2\cdot({N+\beta})\cdot L\cdot\frac{s}{r^\alpha}}\Big]\\
&\leq \frac{A}{r^N}
\cdot  \Big[{1+2\cdot({N+\beta})\cdot L\cdot\frac{s}{r^\alpha}}\Big]^{-1}
\leq \frac{A}{r^N}
\cdot  \Big[{1+L\cdot\frac{s}{r^\alpha}}\Big]^{-(N+\beta)}
= \Big(\frac{r}{z}\Big)^\beta\cdot \frac{A}{z^N}
\end{split}\end{equation*} 
for $s\in[0,\min\{\delta_5 r^\alpha,T\}]$ where $\delta_5:=\min\{\delta_4,(1/L)
\cdot C
\}$.\\

For the $L^1$-condition, from \eqref{ineq_L_1} and
(I) and (III) of Lemma \ref{lem_elemen}, we have
\begin{equation*}\begin{split}
  \|\varphi(s,\cdot)\|_{L^1({\mathbb{R}^N})}
\leq &(1-
 C_{L^1}
\frac{s}{r^{{\alpha}}})
\leq 
1-\beta\cdot L\cdot\frac{s}{r^\alpha}
= \Big(\frac{r}{z}\Big)^\beta 
\end{split}\end{equation*} for $s\in[0,\min\{\delta_5 r^\alpha,T\}]$.\\

%
%
Together with the mean zero property of $\varphi$ in {STEP 1}, we proved
for any $s\in[0,\min\{\delta_5 r^\alpha,T\}]$ with $r\in(0,1]$ and for any
$\varphi_0\in\mathcal{U}_r$, we have the evolution estimate
\begin{equation*}
 \varphi(s) \in 
\Big(\frac{r}{r(1+L\frac{s}{r^\alpha})}\Big)^\beta
\mathcal{U}_{r(1+L\frac{s}{r^\alpha})}=
\Big(\frac{r}{z}\Big)^\beta
\mathcal{U}_{z}.
\end{equation*} which proves \eqref{eq_main_prop}. \\ 

It remains to prove \eqref{eq_main_prop_r_1}.
Let $r=1$ (i.e. $\varphi_0\in\mathcal{U}_1$).
 Note that
Lemma \ref{lem_concent} holds for all time $s\in[0,T)$ 
and $L^p$ norm is decreasing all time  $s\in[0,T)$  and for any $1\leq p\leq\infty$
 from (II) of Lemma \ref{lem_property_equation_with_kernel}. 
Thus we have $\varphi(s) \in (1+Ls)\mathcal{U}_1$ for all
$s\in[0,T)$. This
 is the end of the  proof of Proposition \ref{main_prop}.

\end{proof}

\section{Proof of the part $(II)$ of Theorem \ref{main_thm}}\label{sec_Proof_of_main_thm}
\begin{proof}[Proof of the part $(II)$ of Theorem 
\ref{main_thm}]
Let $t$ be any time between $0$ and $T$.
Thanks to (II) of Lemma \ref{lem_c_beta}
 and 
 (II) of Lemma \ref{lem_property_equation_with_kernel},
 the only thing we need to do is to find an estimate on
$r^{-\beta}\Big|\int_{\mathbb{R}^3}{w}(t,x)\varphi_0(x)dx\Big|$
for  $\varphi_0\in
\mathcal{S}\cap\mathcal{U}_r$ with $0<r\leq 1$.
From Proposition \ref{main_prop}, we have a smooth solution  $\varphi$  on $[0,t]$ correspoding to 
the initial data $\varphi_0$ with the kernel ${{K}^{(t)}}$ 
, which is defined by ${{K}^{(t)}}(s):=K(t-s)$
(see the definition \eqref{def_back_kernel}).
From Lemma \ref{lem_dual}, we want a control on 
$r^{-\beta}\Big|\int_{\mathbb{R}^3}{w}_0(x)\varphi(t,x)dx\Big|$.
  Indeed,
\begin{equation}\begin{split}\label{dual_conclu}
\|{w}(t,\cdot)\|_{C^\beta(\mathbb{R}^N)}&\leq C\cdot \Big(\|{w}(t,\cdot)\|_{L^\infty} + 
\sup_{\varphi_0\in
\mathcal{S}\cap
\mathcal{U}_r,0<r\leq 1}
r^{-\beta}\Big|\int_{\mathbb{R}^3}{w}(t,x)\varphi_0(x)dx\Big| \Big)\\
&\leq C\cdot \Big(\|{w}_0\|_{L^\infty} + 
\sup_{\varphi_0\in
\mathcal{S}\cap
\mathcal{U}_r,0<r\leq 1}
r^{-\beta}\Big|\int_{\mathbb{R}^3}{w}_0(x)\varphi(t,x)dx\Big| \Big).
\end{split}\end{equation}

\begin{rem}\label{rem_repeat}
 The main idea is to repeat \eqref{eq_main_prop} as 
many time as we want until the time evolution reaches the given time $t\in(0,T)$. For example,
as long as $s_1\leq\delta r^\alpha$, $s_2\leq
\delta z(r,s_1)^\alpha $, and $z(r,s_1)\leq 1$, we 
can repeat Proposition \ref{main_prop} twice to 
get
the following time evolution:
\begin{equation*}\begin{split}
 \varphi(0)\in\mathcal{U}_r &\Rightarrow 
\varphi(s_1)\in\Big(\frac{r}{z(r,s_1)}\Big)^\beta\mathcal{U}_{z(r,s_1)}
\end{split}\end{equation*}
\begin{equation*}\begin{split}
\Rightarrow 
\varphi(s_1+s_2)&\in\Big(\frac{r}{z(r,s_1)}\Big)^\beta\times
\Big(\frac{z(r,s_1)}{z(z(r,s_1),s_2)}\Big)^\beta\mathcal{U}_{z(z(r,s_1),s_2)}\\
&=
\Big(\frac{r}{z(z(r,s_1),s_2)}\Big)^\beta\mathcal{U}_{z(z(r,s_1),s_2)}.
\end{split}\end{equation*}
However, 
when $z(
r,s)$ reaches $1$ before the given time $t$, then we cannot use \eqref{eq_main_prop} any more.
Instead, we need to use \eqref{eq_main_prop_r_1}, which
grows as time increases. For this reason, 
we  obtain only \eqref{from c beta to c beta} first
which depends on $t$. 
This defect is overcome
by investigating the evolution of 
the $L^1$ norm of $\varphi(s)$
(see \eqref{from L infty to c beta_before}). 
Since this examination requires 
a careful estimate \eqref{careful_repetition} for repetitions of \eqref{eq_main_prop},
we present a rigours argument  below.
As a result,
the final estimate is independent of the length of time interval
(see \eqref{from c beta to c beta_final}).\\

\end{rem}



Define a constant $\eta:=(1+L\cdot\delta)>1$. For each $r\in(0,1]$,
we define the integer $k=k(r)\geq1$ such that 
$r\cdot\eta^{k-1}\leq 1<r\cdot\eta^{k}$.
Also define $z_n=z_n(r)$ for $n=0,1,2,\cdots,k-1,k$ by
\begin{equation*}z_n=\begin{cases}
&  r\cdot\eta^n
\quad\mbox{ if } 0\leq n\leq (k-1),\\
& 1
\quad\mbox{ if }  n=k.\\
\end{cases}\end{equation*} Note that $r=z_0< z_1 < \cdots< z_{k-1}\leq 1=z_k$.

We find $\tilde{t}=\tilde{t}(r)\in
[0,\delta r^\alpha(\eta^\alpha)^{k-1})
=[0,\delta (z_{k-1})^\alpha)$
such that  ${z_{k-1}(1+L\frac{\tilde{t}}{(z_{k-1})^\alpha})}$ $=1$,
which is always possible because 
$ z_{k-1}\leq 1 <{z_{k-1}\cdot(1+
L\frac{\delta (z_{k-1})^\alpha}{(z_{k-1})^\alpha})}=z_{k-1}\cdot\eta$.

Also define $t_n=t_n(r)$ for $n=0,1,2,\cdots,k-1,k,k+1$ by
\begin{equation*}t_n=\begin{cases}
& \delta\cdot r^\alpha\Big(\frac{(\eta^\alpha)^n-1}{\eta^\alpha-1}\Big) 
\quad\mbox{ if } 0\leq n\leq (k-1),\\
&( t_{k-1}+\tilde{t})
\quad\mbox{ if }  n=k,\\
&\infty \quad\mbox{ if } n=k+1.
\end{cases}\end{equation*} Note that, for $1\leq n\leq(k-1)$,
\begin{equation*}\begin{split}
t_n&=\delta r^\alpha\Big(1+\eta^\alpha+(\eta^\alpha)^2+\cdots+(\eta^\alpha)^{n-1}\Big)\\
&=\delta r^\alpha+\delta r^\alpha\eta^\alpha+\delta r^\alpha(\eta^\alpha)^2+\cdots+\delta r^\alpha(\eta^\alpha)^{n-1}\\
&=\delta(z_0)^\alpha+\delta(z_1)^\alpha+\cdots+\delta(z_{n-1})^\alpha \mbox{ and }\\
\end{split}\end{equation*} 
\begin{equation*}\begin{split}
t_n-t_{n-1}&=\delta(z_{n-1})^\alpha.\\
\end{split}\end{equation*} 
Now we make a partition of $[0,\infty)\subset \mathbb{R}^1$ by
\begin{equation*}
[0,\infty)=[t_0,t_1)\cup[t_1,t_2)\cup\cdots\cup[t_{k-2},t_{k-1})\cup[t_{k-1},t_{k})\cup[t_k,t_{k+1})
\end{equation*} where these union are disjoint.

 Finally, we are ready to  apply the main proposition \ref{main_prop}
as many time as we want. 
Indeed, if $t\in[t_n,t_{n+1})$ with $0\leq n\leq(k-1)$,
then
we can repeat the main proposition \ref{main_prop} so that we obtain
\begin{equation*}\begin{split}
\varphi(t)\in 
&\Big(\frac{r}{z_1}\Big)^\beta\times
\Big(\frac{z_1}{z_2}\Big)^\beta\times
\cdots \times
\Big(\frac{z_n}{z_n(1+L\cdot\frac{t-t_n}{(z_n)^\alpha})}\Big)^\beta
\mathcal{U}_{z_n(1+L\cdot\frac{t-t_n}{(z_n)^\alpha})}\\
&=
\Big(\frac{r}{z_n(1+L\cdot\frac{t-t_n}{(z_n)^\alpha})}\Big)^\beta
\mathcal{U}_{z_n(1+L\cdot\frac{t-t_n}{(z_n)^\alpha})}.\\
\end{split}\end{equation*} 

Moreover, because 
\begin{equation}\begin{split}
\label{varphi_t_k_estimate}
\varphi(t_k)\in 
&\Big(\frac{r}{z_k}\Big)^\beta
\mathcal{U}_{z_k}=\Big(\frac{r}{1}\Big)^\beta
\mathcal{U}_{1},
\end{split}\end{equation} 
we get, for the case 
$t\in[t_k,t_{k+1})=[t_k,\infty)$,
\begin{equation*}\begin{split}
\varphi(t)\in 
&\Big(\frac{r}{1}\Big)^\beta\cdot(1+L(t-t_k))\cdot
\mathcal{U}_{1}.
\end{split}\end{equation*} 
From the above argument,
for any fixed $r\in(0,1]$, we can extend the function $z=z(r,s)$ of Proposition \ref{main_prop}
up to 
all $s\in[0,\infty)$
 by
\begin{equation*}
z(r,s)=
\begin{cases}
&{z_n(1+L\cdot\frac{s-t_n}{(z_n)^\alpha})} \quad\mbox{ for }
s\in[t_n,t_{n+1}) \mbox{ with  } 0\leq n\leq(k-1),\\
&1 \quad\mbox{ for }
s\in[t_k,t_{k+1})=[t_k,\infty).
\end{cases}\end{equation*}
In terms of the function $z$, 
we obtained 
\begin{equation*}\begin{split}
\varphi(t)\in 
& \Big(\frac{r}{z}\Big)^\beta\cdot(1+L\cdot (t-t_k)^+)\cdot
\mathcal{U}_{z}.
\end{split}\end{equation*} 
As a result, we have
\begin{equation*}\begin{split}
\Big|\int_{\mathbb{R}^3}{w}_0(x)\varphi(t,x)dx\Big| 
&\leq \Big(\frac{r}{z}\Big)^\beta \cdot(1+L\cdot t)\cdot z^\beta\cdot 
\|{w}_0\|_{C^\beta(\mathcal{R}^N)}\\
&=r^\beta \cdot (1+L\cdot t)\cdot 
\|{w}_0\|_{C^\beta(\mathcal{R}^N)}.\\
\end{split}\end{equation*} From the observation \eqref{dual_conclu}, we have proved, for any $t\in[0,T)$,
\begin{equation}\label{from c beta to c beta}
      \|{w}(t,\cdot)\|_{C^\beta(\mathbb{R}^N)}
\leq C 
\cdot (1+L\cdot t)
\cdot\|{w}_0\|_{C^\beta(\mathbb{R}^N)}
     \end{equation} where $C$ does not depend on $t$.
     Note that this estimate blows up
     as $t$ goes to infinity.\\

We can overcome the above blow-up defect by obtaining the evolution of $L^1$-norm of 
 $\varphi(s)$. Indeed,
for the case $t<t_k$, i.e. for
$t\in[t_n,t_{n+1})$  with $0\leq n\leq(k-1)$,
 the function $z(r,t)$ is
 bounded below by $C\cdot t^{1/\alpha}$
 where $C$ does not depend on $r\in(0,1]$. Indeed,
\begin{equation}\begin{split}
\label{careful_repetition}
\Big(z(r,t)\Big)^\alpha&=
\Big({z_n(1+L\cdot\frac{t-t_n}{(z_n)^\alpha})}\Big)^\alpha\geq
(z_n)^\alpha=(r\cdot\eta^n)^\alpha=r^\alpha\cdot(\eta^\alpha)^n\\
&\geq
\Big(\frac{\eta^\alpha-1}{\eta\delta}\Big)\delta r^\alpha\cdot
\frac{((\eta^\alpha)^{n+1}-1)}{(\eta^\alpha-1)}
\geq
\Big(\frac{\eta^\alpha-1}{\eta\delta}\Big)t_{n+1}
\geq
\Big(\frac{\eta^\alpha-1}{\eta\delta}\Big)t.\\
\end{split}\end{equation}
Recall that  
$
\varphi(t)\in 
 ({r}/{z})^\beta\cdot
\mathcal{U}_{z}
$ for any $t<t_k$. Thus, thanks to
\eqref{careful_repetition}, we have the evolution of $L^1$-norm 
$\|\varphi(t)\|_{L^1(\mathbb{R}^N)}\leq ({r}/{z})^\beta\leq
C\cdot{r}^\beta\cdot {t^{-\beta/\alpha}} $.\\
On the other hand, from \eqref{varphi_t_k_estimate}, 
we have $\|\varphi(t_k)\|_{L^1(\mathbb{R}^N)}\leq r^\beta $. Thus,
from (II) of Lemma \ref{lem_property_equation_with_kernel}, we get
$\|\varphi(t)\|_{L^1(\mathbb{R}^N)}\leq r^\beta $ as long as $t\geq t_k$.
Therefore, we have a control for any $t\in(0,T)$:
\begin{equation}\begin{split}\label{from L infty to c beta_before}
\Big|\int_{\mathbb{R}^3}{w}_0(x)\varphi(t,x)dx\Big| 
&\leq 
\|{w}_0\|_{L^{\infty}(\mathcal{R}^N)}\cdot
\|\varphi(t)\|_{L^1(\mathbb{R}^N)}\\
&\leq
C \cdot r^\beta\cdot{\max\{1,\frac{1}{t^{\beta/\alpha}}\}}\cdot
\|{w}_0\|_{L^{\infty}(\mathcal{R}^N)}.
\end{split}\end{equation} Thus, from
\eqref{dual_conclu}, we have
\begin{equation}\label{from L infty to c beta}
      \|{w}(t,\cdot)\|_{C^\beta(\mathbb{R}^N)}
\leq C \cdot{\max\{1,\frac{1}{t^{\beta/\alpha}}\}}\cdot\|{w}_0\|_{L^\infty(\mathbb{R}^N)} 
     \end{equation} where $C$ does not depend on $t$. 
 Now we can combine
   \eqref{from c beta to c beta}
with \eqref{from L infty to c beta}
to get
\begin{equation}\begin{split}\label{from c beta to c beta_final}
      \|{w}(t,\cdot)\|_{C^\beta(\mathbb{R}^N)}
& \leq C \cdot\min\{{\max\{1,\frac{1}{t^{\beta/\alpha}}\}}, 
(1+L\cdot t)
\}
\cdot\|{w}_0\|_{C^\beta(\mathbb{R}^N)}\\
&\leq C 
\cdot\|{w}_0\|_{C^\beta(\mathbb{R}^N)} 
\mbox{ for any } t\in[0,T).
   \end{split}  \end{equation}

Similarly, we can prove
$\|\varphi(t)\|_{L^\infty(\mathbb{R}^N)}
\leq
C \cdot r^\beta\cdot{\max\{1,\frac{1}{t^{(N+\beta)/\alpha}}\}}$. As a result,
from
\eqref{dual_conclu},
we have 
\begin{equation}\label{from L 1 to c beta}
      \|{w}(t,\cdot)\|_{C^\beta(\mathbb{R}^N)}
\leq C \Big(
\|{w}_0\|_{L^\infty(\mathbb{R}^N)}+
{\max\{1,\frac{1}{t^{(N+\beta)/\alpha}}\}}\cdot\|{w}_0\|_{L^1(\mathbb{R}^N)}\Big).
     \end{equation}
\end{proof}

\section{Appendix}\label{sec_Appendix}

\subsection{Proof of the part $(I)$ of
Theorem \ref{main_thm}}
\begin{proof}[Proof of the part $(I)$ of Theorem 
\ref{main_thm}]
In this subsection, we suppose that the kernel 
$K$  satisfies not
the $weak$-$(*)$-kernel condition
in Definition \ref{def_parameter2}
but the $(*)$-kernel condition
in Definition \ref{def_parameter} (we recall
that the latter condition implies
the former one). Note that
the kernel $K$ does not need
to satisfy \eqref {smooth_assump_lin}
any more.
Thus, we first    construct a family of kernels
$K_\epsilon$
keeping all the parameters 
of the $(*)$-kernel condition uniformly in $\epsilon>0$,
and satisfying
\eqref{smooth_assump_lin}. Then
we   use
the conclusion of the part $(II)$ of Theorem 
\ref{main_thm}.\\

We define $\Phi$ by $\Phi(t,x,y):=\Phi^1(t)\Phi^2(x)\Phi^2(y)$
for $t\in\mathbb{R}$ and $x,y\in\mathbb{R}^N$ 
and $\Phi_\epsilon(\cdot):=\epsilon^{-(2N+1)}\Phi(\cdot/\epsilon)$ 
where
$\Phi^1$ and $\Phi^2$ are standard $C_c^\infty$ mollifiers in 
$\mathbb{R}$ and $\mathbb{R}^N$, respectively. 
Let $(w_0)_\epsilon:=w_0*\Phi^2$ and $h(t,x,y):=k(t,x,y-x)$.
Then we define a family of kernels  by $h_\epsilon:=h^{\epsilon}*_{t,x,y}\Phi_\epsilon$
where \begin{equation*}
{
h^\epsilon(t,x,y):=\begin{cases} & h(t,x,y)\quad\mbox{ for } 
{
|x-y|<(1/\epsilon)} \mbox{ with }t\in[0,\min\{(1/\epsilon),T\}],\\
&
{\Lambda^{-1}} 
\quad \quad\quad\mbox{ otherwise}.\end{cases}
}
\end{equation*}
Since $|h^\epsilon(t,x,y)|\leq \Lambda(1+\epsilon^{-\omega})<\infty$
for all $t,x$ and $y$, 
we  observe that $k_\epsilon(t,x,z):=h_\epsilon(t,x,x+z)$ satisfies the condition \eqref{smooth_assump_lin}.

For each 
$\epsilon\ll\zeta/2$,
the associated kernel $K_\epsilon(t,x,y)
:=k_\epsilon(t,x,y-x)\cdot|y-x|^{-(N+\alpha)}$ 
satisfies the $(*)$-kernel condition on the same parameter set  of the original kernel $K$
except $\zeta_\epsilon:=\zeta/2$ and
$\Lambda_\epsilon:=2\Lambda$ (for $\alpha\geq1$,
we assume further
$\epsilon\ll s_0/2$ and $(s_0)_\epsilon:=s_0/2$).
Then, we can construct a weak solution $w_\epsilon$ corresponding
the kernel $K_\epsilon$ and the initial data $(w_0)_\epsilon$, 
and this solution $w_\epsilon$ is smooth
 since $k_\epsilon$
satisfy \eqref{smooth_assump_lin}
(for existence, see  \cite{komatsu_fundamental} or refer the approximation scheme
 in \cite{ccv} while  smoothness is a consequence of 
a standard energy argument). 
Thanks to the part $(II)$ of Theorem \ref{main_thm},
these solutions satisfy 
\eqref{eq_C_beta_main_thm},
 \eqref{eq_L_infty_main_thm}, and 
 \eqref{eq_L_1_main_thm}. As a result,
 we can extract a limit function $w$, which is a
 weak solution for the original kernel $K$ and the initial 
 data $w_0$.
\end{proof}
\subsection{Proof of Theorem 
\ref{main_thm_nonlinear}}\label{prof_non_linear}
\begin{proof}[Proof of Theorem \ref{main_thm_nonlinear}]

For convenience, we define a function $g$ 
by $g(x)=G(x)\cdot|x|^{N+{\alpha}}$.
In addition to all the assumptions 
of Theorem \ref{main_thm_nonlinear}, we assume further  
\begin{equation}\label{smooth_assump_nonlin}
\theta_0\in C^\infty(\overline{\mathbb{R}^N}),\quad g\in C^\infty(\overline{\mathbb{R}^N}), \mbox{ and } \phi\in C^\infty(\overline{\mathbb{R}}).
\end{equation} Then there exists a weak solution $\theta$ of \eqref{main_non_lin_eq} in global time
and it is smooth. Indeed,
for existence issue, we refer to
Benilan  and Brezis \cite{MR0336471}
or the appendix in the paper \cite{ccv}.
Smoothness follows a
difference quotient argument.\\

We will show that the conclusions of Theorem \ref{main_thm_nonlinear} hold for this smooth solution $\theta$.
Moreover, it will be clear that the constants $C$ and $\beta$ depend only on
the parameters in the hypotheses of Theorem \ref{main_thm_nonlinear} and they
are independent of the actual norms coming from the above additional assumption \eqref{smooth_assump_nonlin}.
Thus the conclusions  of Theorem \ref{main_thm_nonlinear}  without  \eqref{smooth_assump_nonlin}
follows
by a limit argument.\\

\begin{rem}
Indeed, if we do not have \eqref{smooth_assump_nonlin}, then we  regularize
$\theta_0,g, $ and $ \phi$ first: 
\begin{equation*}
(\theta_0)_\epsilon:=\theta_0*\Phi^2_\epsilon,\quad
 g_\epsilon:=g^\epsilon*\Phi^2_\epsilon,\quad \mbox{ and }
  \phi_\epsilon:=\phi*\Phi^1_\epsilon
\end{equation*} where $\Phi^1$ and $\Phi^2$ are mollifiers in $\mathbb{R}^1$ and $\mathbb{R}^N$, respectively,
and $g^\epsilon$ is defined by $g^\epsilon(x):=\begin{cases}&g(x)\mbox{ if }
|x|\leq(1/\epsilon),\\&0 \mbox{ otherwies}.\end{cases}$ As a result, we obtain 
\eqref{smooth_assump_nonlin} for $(\theta_0)_\epsilon, g_\epsilon, $ and $ \phi_\epsilon$.
Moreover, for any $\epsilon\leq (\zeta/2)$, all the assumptions (the parameters) 
of Theorem \ref{main_thm_nonlinear}
still work for for $(\theta_0)_\epsilon, g_\epsilon, $ and $ \phi_\epsilon$
except we need to replace
the original $\zeta$ by $\zeta/2$ for the condition \eqref{cond_bounds_kernel_nonlinear}.\\
\end{rem}

We take a derivative $(D_e\theta:=w)$ on the equation \eqref{main_non_lin_eq}
so that we get the following equation  
\begin{equation*}
 \partial_t w(t,x)-
\int_{\mathbb{R}^N}
(w(t,y)-w(t,x))
\phi^{\prime\prime}(\theta(t,y)-\theta(t,x))G(y-x)dy=0.
\end{equation*} By putting 
$K(t,x,y):=\phi^{\prime\prime}(\theta(t,y)-\theta(t,x))G(y-x)$,
this function $w(=D_e\theta)$ solves the linear equation \eqref{main_lin_eq}.
Moreover,
it is easy to see that this new kernel $K$ satisfies
\eqref{cond_symm},  
\eqref{cond_bounds}, 
and \eqref{smooth_assump_lin}  directly
(a rigorous proof can be completed by using the difference quotient argument, which 
is  contained in \cite{ccv}).
Then, Theorem \ref{main_thm_nonlinear} for the case $\alpha<1$ follows once we
apply  the part $(II)$ of Theorem 
\ref{main_thm} to $w$.\\

For the the case $\alpha\geq1$, we need to verify 
the cancellation condition \eqref{cond_cancell}
to get the $weak$-$(*)$-kernel condition.
Let $s\in(0,1)$, $t\in[0,T]$ and $ x\in\mathbb{R}^N$. Then, we have
\begin{equation*}\begin{split}
\Big| \int_{S^{N-1}} 
k(t,x,s\sigma)\sigma d\sigma\Big|&=
\Big| \int_{S^{N-1}} 
K(t,x,x+s\sigma)|s\sigma|^{N+\alpha}\sigma d\sigma\Big|\\
&=
\Big| \int_{S^{N-1}_+} 
\phi^{\prime\prime}(\theta(t,x+s\sigma)-\theta(t,x))G(s\sigma)s^{N+\alpha}
\sigma d\sigma\\
&\quad\quad
+
\int_{S^{N-1}_-} 
\phi^{\prime\prime}(\theta(t,x+s\sigma)-\theta(t,x))G(s\sigma)s^{N+\alpha}
\sigma d\sigma
\Big|\\
\end{split}\end{equation*} where $S^{N-1}_+$ and $S^{N-1}_-$ are upper and lower
hemispheres, respectively. Then, by symmetry of $G(\cdot)$,
\begin{equation*}\begin{split}
&=
\Big| \int_{S^{N-1}_+} 
\Big[\phi^{\prime\prime}(\theta(t,x+s\sigma)-\theta(t,x))
-\phi^{\prime\prime}(\theta(t,x-s\sigma)-\theta(t,x))
\Big]G(s\sigma)s^{N+\alpha}
\sigma d\sigma.\\
\end{split}\end{equation*} We use the assumption $\phi^{\prime\prime}\in C^\nu$:
\begin{equation*}\begin{split}
&\leq
 \int_{S^{N-1}_+} 
[\phi^{\prime\prime}]_{C^{\nu}(\mathbb{R})}\cdot\Big|\theta(t,x+s\sigma)
-\theta(t,x-s\sigma)
\Big|^{\nu}
G(s\sigma)s^{N+\alpha}
|\sigma| d\sigma\\
&\leq
 \int_{S^{N-1}_+} 
[\phi^{\prime\prime}]_{C^{\nu}(\mathbb{R})}\cdot\|\nabla\theta(t)
\|^{\nu}_{L^\infty_{x}}
\Big|2s\sigma
\Big|^{\nu}
G(s\sigma)s^{N+\alpha}
 d\sigma\\
&\leq C
[\phi^{\prime\prime}]_{C^{\nu}(\mathbb{R})}
\cdot\|\nabla\theta_0\|^{\nu}_{L^\infty}\cdot\sqrt{\Lambda}
 \cdot s^{\nu}\cdot\int_{S^{N-1}_+} 
(1+s^\omega)
 d\sigma\\
&\leq  C
\cdot
 M \cdot\sqrt{\Lambda}\cdot
s^{\nu}\cdot(1+s^\omega)
\leq  C
\cdot
 M \cdot\sqrt{\Lambda}\cdot
s^{\nu}
\end{split}\end{equation*} where
 the  proof of non-increasing of 
$\|\nabla\theta(t)\|_{L^\infty_{x}}$
is in the part $(II)$ of  Lemma \ref{lem_property_equation_with_kernel}.
By putting $\tau:=C
\cdot M\cdot\sqrt{\Lambda}$ with $s_0:=1$,
we get the condition \eqref{cond_cancell}.
Then, we apply  the part $(II)$ of
Theorem \ref{main_thm} to $w$.

\end{proof}

\bibliographystyle{plain}
\bibliography{bib_Kyudong_Choi_2011Dec28}

\end{document}